 \documentclass[final,1p,times,10pt]{elsarticle}

\usepackage[usenames,dvipsnames]{color}

\usepackage{graphicx}
\usepackage{epsfig}
\usepackage{amsmath, amssymb, amsthm, amstext}
\usepackage{todonotes}
\usepackage{subcaption}
\usepackage{multirow}
\usepackage{bbm, bm}
\usepackage{diagbox}
\usepackage{array}

\newtheorem{thm}{Theorem}
\newtheorem{lem}{Lemma}
\newtheorem{cor}{Corollary}

\theoremstyle{definition}
\newtheorem{rem}{Remark}

\newtheorem{ex}{Example}
\newtheorem*{problem*}{Problem}
\newtheorem*{assumption*}{Assumption}
\newtheorem*{exs*}{Examples}
\newtheorem*{ag*}{Assumptions and Goal}
\newtheorem*{ag2*}{Assumption}

 \biboptions{square,numbers}

\usepackage{lipsum}
\usepackage{mathtools}
\usepackage{cancel}
\usepackage{enumerate}

\usepackage{hyperref}

\allowdisplaybreaks
\begin{document}

\begin{frontmatter}

\title{Non-ergodic inference for stationary-increment harmonizable stable processes}

\author{Ly Viet Hoang}
\ead{ly.hoang@uni-ulm.de}
\author{Evgeny Spodarev}
\ead{evgeny.spodarev@uni-ulm.de}
\address{Ulm University}

\journal{}

\begin{abstract}
We consider the class of stationary-increment harmonizable stable processes with infinite control measure, which most notably includes real harmonizable fractional stable motions.
We give conditions for the integrability of the paths of such processes with respect to a finite, absolutely continuous measure and derive the distributional characteristics of the path integral with respect to said measure. 
The convolution of the path of a stationary-increment harmonizable stable process with a suitable measure yields a real stationary harmonizable stable process with finite control measure. 
This allows us to construct consistent estimators for the index of stability as well as the kernel function in the integral representation of a stationary increment harmonizable stable process (up to a constant factor). 
For real harmonizable fractional stable motions consistent estimators for the index of stability and its Hurst parameter are given. These are computed directly from the periodogram frequency estimates of the smoothed process. 
\end{abstract}

\begin{keyword}

{Fourier analysis}
\sep{fractional stable motion}
\sep{frequency estimation}
\sep{Hurst parameter}
\sep{harmonizable process}
\sep{non-ergodic process}
\sep{non-ergodic statistics}
\sep{periodogram}
\sep{spectral density estimation}
\sep{stable process}
\sep{process with stationary increments}

\end{keyword}

\end{frontmatter}


\section{Introduction}

A stochastic process or random field $X=\{X(t):t\in T\}$ is called \emph{self-similar} with index $H$ or $H$-self-similar  if it is invariant in distribution under scaling of time and space.
Specifically, the processes $\{X(at):t\in T\}$ and $\{a^HX(t):t\in T\}$ are equal in distribution, meaning that all their finite dimensional distributions are equal.
Furthermore, $X$ has \emph{stationary increments} if for all $h\in T$ the processes $\{X(t+h)-X(h):t\in T\}$ and $\{X(t)-X(0):t\in T\}$ have identical finite-dimensional distributions. 
A classical example of a self-similar stochastic process with stationary increments is the fractional Brownian motion, a process widely used for modeling phenomena in finance, telecommunications or natural sciences such as hydrology since self-similar objects (or fractals)  are abundant in nature  \cite{brownian,kolmogorovfbm, lamperti1962semi, fbm, mandelbrotwallis1,mandelbrotwallis2,mandelbrotwallis3}.
	
Fractional Brownian motions are characterized by their \emph{Hurst parameter} or \emph{exponent} $H\in(0,1)$ and generalize the classical Brownian motion or Wiener process ($H=0.5$).
In fact, the roughness of a fractional Brownian motion's path, its memory as well as the dependency between its increments are determined by the Hurst parameter.
Moreover, fractional Brownian motions are self-similar with index $H$ \cite{selfsimilarprocesses}.
Furthermore, practical problems such as modeling of networks show that besides self-similarity heavy tailed distributions are of crucial importance as well \cite{fbm2, fbm3}. Hence, the extension to stable distributions is natural. 

A fractional Brownian motion admits both a moving average as well as a harmonizable representation \cite[Chapter 7.2]{gennady}. 
The generalization of its moving average representation to the $\alpha$-stable regime with $0<\alpha<2$ yields the class of linear fractional stable motions (LFSMs), whereas the extension of the harmonizable representation results in the class of harmonizable fractional stable motions (HFSMs) \cite{cambanismaejima,samorodnitskytaqquLFSM}, see also \cite{maejima} and \cite[Chapter 7.4, 7.7]{gennady} for an overview. 
It is important to emphasize that the class of LFSMs and the class of HFSMs are disjoint for $0<\alpha<2$. Most notably, HFSMs are non-ergodic \cite{cambanis1}, 
hence statistical inference based on empirical averages is not possible.

Statistical inference of fractional Brownian motions has been studied extensively in \cite{dahlhaus,foxtaqqu, mandelbrottaqqu}, see  \cite{beran2017statistics} for an overview. 
A selection of works on the statistical inference for LFSMs is given by \cite{ayache3,ayache1,podolskij3,podolskij2,podolskij,pipirasstoevtaqqu}, whereas HFSMs have not received as much attention in this regard due to their non-ergodicity. 
Only in recent years a handful of theoretical results on HFSMs have been published. 
Most notably, \cite{oconnorpodolskij} presents the asymptotic theory for quadratic variations of HFSMs, \cite{ayache2,ayachexiao,ayachexiao2} show path properties and a series representation for HFSMs, whereas \cite{bierme} considers the more general class of operator scaling harmonizable random fields. 

None of the above works on HFSMs consider point estimation of the stability index and Hurst parameter. Solely the preprint \cite{ayache} proposes a framework for the simultaneous estimation of the index of stability and Hurst parameter based on wavelet transformations of the path of a HFSM, from which independent symmetric $\alpha$-stable random variables are obtained. Their scale paramaters are closely related to the index of stablity and Hurst parameter. 
Although appealing in theory, this method relies on knowledge of the whole path of a HFSM on the real line since the required wavelets are oscillating functions at their tails, which is unfeasible from a practical point of view. 

Throughout this paper we consider $\alpha\in(0,2)$.
We consider the broader class of stationary-increment harmonizable processes $X=\{X(t):t\in\mathbb{R}\}$ with integral representation 
\begin{align*}
	X(t)=Re\left(\int_{\mathbb{R}}\left(e^{itx}-1\right)\Psi(x)M_\alpha(dx)\right),\quad t\in\mathbb{R},
\end{align*}
where $M_\alpha$ is a complex isotropic symmetric $\alpha$-stable random measure with Lebesgue control measure.
The spectral density $\Psi\in L^{\alpha}(\mathbb{R},\min(1,\vert x\vert^\alpha)dx)$ characterizes the process uniquely. 
For the special choice $\Psi(x)=\vert x\vert^{-H-1/\alpha}$ the process $X$ is a real HFSM. 
In this paper, we give consistent estimators for the index of stability $\alpha$ and the kernel function $\Psi$ (up to a constant factor). 
In case of real HFSMs, consistent estimators for $\alpha$ as well as the Hurst parameter $H$ are obtained.
Our method is motivated by \cite{ayache} but relies only on suitable convolutions of a single path of $X$ making it effectively a smoothing procedure. 
There are only minimal constraints on the smoothing kernel (the so-called mollifier). 
Thus, choosing a mollifier with exponentially decaying tails is possible, ensuring that our method is highly practical. 

The paper is structured as follows. 
In Section \ref{section: prelim}, we  give basic notation and definitions on $\alpha$-stable laws. 
Furthermore, complex $S\alpha S$ random measures and processes are introduced.
The class of stationary real harmonizable stable processes with finite control measure and its properties are discussed, and the challenges of statistical inference on these processes are highlighted. 
We also briefly introduce the periodogram frequency and spectral density estimation for stationary real harmonizable stable processes from \cite{viet2}.
The section concludes with posing the problem of estimating the index of stability $\alpha$ and spectral density $\Psi$ for harmonizable stable processes with infinite control measure, in particular for real harmonizable fractional stable motions.

We consider the mollified or smoothed version of a stationary-increment harmonizable stable process in Section \ref{section: mollifier}. 
Conditions on the smoothing kernel for a well-defined smoothing procedure are given. In fact, the smoothing kernel can be chosen such that the resulting smoothed process is stationary real harmonizable $S\alpha S$ process with finite control measure. 

Section \ref{section: statistical inference} focuses on the statistical inference on stationary-increment harmonizable stable processes as well as real harmonizable fractional stable motions. 
We first consider the more general class of stationary-increment harmonizable stable processes in Section \ref{section: stat incr proc}. We give a strongly consistent estimator for the index of stability $\alpha$ based on two differently smoothed versions of the same path. Kernel density estimation on the frequencies in each mollified process, respectively, as well as linear regression on the logarithm of the ratio of the two kernel density estimators lead to the assessment of $\alpha$. Furthermore, a consistent estimator for the spectral density $\Psi$ up to a constant factor is given.

The special case of real harmonizable fractional stable motions is then considered in Section \ref{section: fractionalmotion}. 
We give a class of exponentially decaying smoothing kernels yielding a mollified process which is again stationary real harmonizable $S\alpha S$. 
The absolute $p$-th power (where $p\geq 1$ coincides with the rate of decay of the smoothing kernel at the origin) of the underlying random frequencies in the LePage series representation can be shown to be i.i.d. Gamma distributed. 
This allows us to compute the index of stability $\alpha$ and the Hurst parameter $H$ directly from a combination of periodogram frequency estimation and consistent inference on the parameters of a Gamma distribution.  

Lastly, we give an extensive numerical analysis in a simulation study in Section \ref{section: numerical}
and conclude our work with the discussion in Section \ref{section: conclusion}.

\section{Preliminaries}
\label{section: prelim}

The following notation is used throughout this work. 
For any complex number $z\in\mathbb{C}$ we denote its real and imaginary part by $Re(z)=z^{(1)}$ and $Im(z)=z^{(2)}$, respectively. 
The notation $\vert\cdot\vert$ is both used for the absolute value on $\mathbb{R}$ as well as on $\mathbb{C}$, where $\vert z\vert=\left((z^{(1)})^2+(z^{(2)})^2\right)^{1/2}$ for $z\in\mathbb{C}$. 

Let $(E,\mathcal{E},\mu)$ be a measurable space. We define the space $L^p(E,\mu)$ with $0< p <\infty$, i.e. the space $p$-integrable functions with respect to a measure $\mu$ on some measurable space $E$, by
\begin{align*}
	L^p(E,\mu)=\left\{u:E\rightarrow\mathbb{C}: \int_E\vert u(x)\vert^p \mu(dx)<\infty\right\}.
\end{align*}
We write $\Vert u\Vert_{L^p(E,\mu)}=\left(\int_E \vert u(x)\vert^p\mu(dx)\right)^{1/p}$. 
Furthermore, we denote the space of $p$-integrable functions on $\mathbb{R}$ with respect to the Lebesgue measure on $\mathbb{R}$ simply by $L^p(\mathbb{R})$ and write $\Vert u\Vert_p=\left(\int_\mathbb{R}\vert u(x)\vert^pdx\right)^{1/p}$.
The notation $\Vert\cdot\Vert_\infty$ stands for the uniform norm. 

Denote by $\mathcal{L}^0(\Omega)$ the space of real-valued random variables on the probability space $(\Omega,\mathcal{F},P)$. 
Additionally, let $\mathcal{L}_c^0(\Omega)=\{X=X^{(1)}+iX^{(2)}: X^{(1)},X^{(2)}\in \mathcal{L}^0(\Omega)\}$ be the space of complex-valued random variables. 
We write $X\stackrel{d}{=}Y$ if the random elements $X$ and $Y$ are equal in distribution. 
In particular, for random processes (or fields) this  means equality of their finite dimensional distributions. 

A random variable $X\in\mathcal{L}^0(\Omega)$ is \emph{symmetric $\alpha$-stable} ($S\alpha S$) with \emph{index of stability} $\alpha\in(0,2]$ and \emph{scale parameter} $\sigma>0$, i.e. $X\sim S\alpha S(\sigma)$, if its characteristic function is given by 
\begin{align*}
	\mathbb{E}\left[e^{isX}\right]=\exp\left\{-\sigma^\alpha\vert s\vert^\alpha\right\}.
\end{align*}
Note that in the case $\alpha=2$ the random variable $X$ is Gaussian with mean $0$ and variance $2\sigma^2$. 
Additionally, a $\mathbb{R}^n$-valued random vector $\bm{X}=\left(X_1,\dots,X_n\right)$ is $S\alpha S$ if its joint-characteristic function is of the form
\begin{align*}
	\mathbb{E}\left[\exp\left\{i(\bm{s},\bm{X})\right\}\right]=\exp\left\{-\int_{S^{n-1}}\left\vert(\bm\theta,\bm{s})\right\vert^\alpha\Gamma(d\bm{\theta})\right\},
\end{align*}
where $S^{n-1}$ is the Euclidean unit sphere in $\mathbb{R}^n$ and $(\cdot,\cdot)$ denotes the scalar product of two vectors in $\mathbb{R}^n$.
The measure $\Gamma$ on $S^{n-1}$ is the \emph{spectral measure} of $\bm{X}$. This measure is unique, finite and symmetric for $0<\alpha<2$ \cite[Theorem 2.4.3]{gennady}.  
A random variable $X\in \mathcal{L}_c^0(\Omega)$ is said to be $S\alpha S$ if $\left(X^{(1)},X^{(2)}\right)=\left(Re(X),Im(X)\right)$ is a  $S\alpha S$ random vector.

Consider the measurable spaces $(E,\mathcal{E})$ and $(S^1,\mathcal{B}(S^1))$, where $\mathcal{B}(S^1)$ is the Borel-$\sigma$-algebra on the unit circle $S^1$.
For a measure $k$ on the product space $(E\times S^1,\mathcal{E}\times\mathcal{B}(S^1))$, define the space $\mathcal{E}_0=\left\{A\in\mathcal{E}:k(A\times S^1)<\infty\right\}$.
A \emph{complex $S\alpha S$ random measure} on $(E,\mathcal{E})$ is $\sigma$-additive, complex-valued set function $M_\alpha:\mathcal{E}_0\rightarrow \mathcal{L}_c^0(\Omega)$, mapping into the space of complex random variables, such that $M_\alpha$ is independently scattered, i.e. $M_\alpha(A)$ and $M_\alpha(B)$ are independent for $A,B\in\mathcal{E}_0$ disjoint, and $M_\alpha(A)$ is a complex $S\alpha S$ random variable for all $A\in\mathcal{E}_0$. 
The $S\alpha S$ random vector $(M_\alpha^{(1)}(A),M_\alpha^{(2)}(A))$ has spectral measure $\Gamma_A=k(A\times\cdot)$ on $(S^1,\mathcal{B}(S^1))$. 
The measure $k$ is called \emph{circular control measure}, whereas the measure $m=k(\cdot\times S^1)$ on $(E,\mathcal{E}_0)$ is referred to as \emph{control measure} of $M_\alpha$ \cite[Definition 6.1.2]{gennady}.

It holds that integration with respect to a $S\alpha S$ random measure $M_\alpha$ of the form $I(u)=\int_E u(x)M_\alpha(dx)$ is well-defined for $u\in L^\alpha(E,m)$ \cite[Chapter 6.2]{gennady}.
Setting $(E,\mathcal{E})=(\mathbb{R},\mathcal{B}(\mathbb{R}))$, where $\mathcal{B}(\mathbb{R})$ is the Borel $\sigma$-algebra on $\mathbb{R}$, allows us to define complex-valued $S\alpha S$ random processes.
For any $t\in\mathbb{R}$ let $h_t(\cdot)$ be a complex-valued function in $L^\alpha(\mathbb{R},m)$. 
Then, the complex-valued $S\alpha S$ process $X=\{X(t):t\in\mathbb{R}\}$ is defined by its integral representation 
\begin{align}
	\label{complexfield}
	X(t)=\int_{\mathbb{R}}h_t(x)M_\alpha(dx),\quad t\in\mathbb{R}, 
\end{align}
where $M_\alpha$ is a complex $S\alpha S$ random measure.

In particular, the case of complex  \emph{isotropic} $S\alpha S$ measures is of interest. 
The random measure $M_\alpha$ is isotropic if $e^{i\phi}M_\alpha(A)\stackrel{d}{=}M_\alpha(A)$ holds for any $\phi\in[0,2\pi)$ and $A\in\mathcal{E}_0$. 
Furthermore, it can be shown that a complex $S\alpha S$ random measure $M_\alpha$ is isotropic if and only if its circular control measure is of the form $k=m\cdot\gamma$, where $\gamma$ is the uniform probability measure on $S^1$ \cite[Example 6.1.6]{gennady}.
Consequently, the measure $M_\alpha$ and therefore the process $X$ are uniquely characterized by control measure $m$. 

\subsection{Stationary real harmonizable $S\alpha S$ processes with finite control measure}

The stochastic processes $X=\{X(t):t\in\mathbb{R}\}$ with
\begin{align}
	\label{SRH}
	X(t)=Re\left(\int_\mathbb{R}e^{itx}M_\alpha(dx)\right),
\end{align}
where $M_\alpha$ is a complex isotropic $S\alpha S$ random measure on $(\mathbb{R},\mathcal{B})$ with finite control measure $m$, 
is called a \emph{stationary real harmonizable $S\alpha S$ process}. 
The isotropy of the random measure $M_\alpha$ is a sufficient and necessary condition for the stationarity of the process $X$ \cite[Theorem 6.5.1]{gennady}, hence the control measure $m$ uniquely determines $M_\alpha$. 
In particular, the distribution of the process $X$ is fully characterized by $m$. 
Indeed, it holds that the finite-dimensional characteristic function of $X$ for all $n\in\mathbb{N}$,
$s_1,\dots,s_n\in\mathbb{R}$ and $t_1,\dots,t_n\in\mathbb{R}$ is given by 
\begin{align}
	\label{charfn}
	\mathbb{E}\left[\exp\left\{i\sum_{i=1}^ns_iX(t_i)\right\}\right]=\exp\left\{-\lambda_\alpha\int_\mathbb{R}\left\vert\sum_{j,k=1}^ns_js_k\cos\left(\left(t_k-t_j\right)x\right)\right\vert^{\alpha/2}m(dx)\right\}
\end{align}
with constant
$\lambda_\alpha=\frac{1}{2\pi}\int_0^{2\pi}\vert\cos\left(x\right)\vert^\alpha dx$ \cite[Proposition 6.6.3]{gennady}.

Assume that the control measure $m$ is absolutely continuous with respect to the Lebesgue measure on $\mathbb{R}$, i.e. it holds that $m(dx)=f(x)dx$. We refer to the Radon-Nikodym derivative $f$ as the \emph{spectral density} of the measure $M_\alpha$ (or equivalently of the process $X$). 
Clearly, under this assumption the spectral density uniquely characterizes the process $X$ as can be seen from the finite-dimensional characteristic function of $X$ in Equation \eqref{charfn}. 
The main goal is to estimate the spectral density from a single path of $X$.

Statistical inference on stationary real harmonizable $S\alpha S$ processes is by far not trivial due to their non-ergodic nature. 
Note that the class of stationary harmonizable $S\alpha S$ processes and stationary $S\alpha S$ moving averages is disjoint for $0<\alpha<2$ \cite[Theorem 6.7.2]{gennady}, see also \cite{cambanissoltani} for an in-depth account on the harmonizable and moving average representation of stationary stable processes as well as \cite{rosinski} for details on the their general structure. 
In fact, only in the Gaussian case ($\alpha=2$) a stationary $S\alpha S$ process might admit both a harmonizable as well as a moving average representation. 

For $0<\alpha<2$ the crucial difference between both classes lies in their ergodicity. In fact, stationary $S\alpha S$ moving averages are ergodic whereas stationary harmonizable $S\alpha S$ processes are non-ergodic \cite{cambanis1}.
When a stationary stochastic process is non-ergodic, the use of standard estimation methods involving empirical averages is not feasible as the limit is random. 
This is due to Birkhoff's ergodic theorem \cite[Corollary 10.9]{kallenberg}, see also \ref{appendix: ergodicity} for brief account on the ergodicity of stationary real  processes.

\subsection{Problem setting}
\label{subsection: problem setting}

The following problem can be stated for random fields on $\mathbb{R}^d$. For numerical reasons we stick to the one-dimensional case $d=1$, see Remark \ref{rem: Rd periodogram}.
Extending the results of \cite{viet2}, the main focus of this work is the statistical inference on harmonizable $S\alpha S$ processes similar to Equation \eqref{SRH} with an \emph{infinite control measure}, i.e. when $m(\mathbb{R})=\infty$. 
If the control measure $m$ is absolutely continuous with respect to the Lebesgue measure on $\mathbb{R}$, then this means that the spectral density $f(x)=m(dx)/dx$ is not integrable on $\mathbb{R}$.

First, note that the process  $$X(t)=Re\left(\int_{\mathbb{R}}e^{itx}M_\alpha(dx)\right),\quad t\in\mathbb{R},$$ is not well defined when $M_\alpha$ has infinite control measure since 
\begin{align*}
	\int_\mathbb{R}\left\vert e^{itx}\right\vert^\alpha m(dx)=\int_\mathbb{R}m(dx)=m(\mathbb{R})=\infty. 
\end{align*}
and the necessary condition $e^{itx}=h_t(x)\in L^\alpha(\mathbb{R},m)$ is not fulfilled.
Instead, we may consider harmonizable $S\alpha S$ processes of the form 
\begin{align}
	\label{newharmonizable}
	X(t)=Re \left(\int_\mathbb{R}\left(e^{itx}-1\right) M_\alpha(dx)\right), 
\end{align}
where again $M_\alpha$ is an isotropic complex $S\alpha S$ random measure with infinite control measure $m(dx)=f(x)dx$. A sufficient condition for $X$ to be well-defined is that the spectral density $f$ lies in the weighted Lebesgue space $L^1(\mathbb{R},\min(1,\vert x\vert^\alpha)dx)$. 
This follows from the simple fact that $\vert e^{itx}-1\vert^\alpha=2\vert\sin(tx/2)\vert^\alpha$, which is bounded for all $t,x\in\mathbb{R}$ and behaves like like $\vert x\vert^\alpha$ as $x\rightarrow0$ for any $t\in\mathbb{R}$.
Again, the goal here is to estimate the index of stability $\alpha$ as well as the spectral density $f$ from a single path of $X$. 

Alternatively, when the control measure $m$ is the Lebesgue measure on $\mathbb{R}$, statistical inference for the spectral density is obviously not of interest as $f\equiv 1$.  
In this case, we study the class of stationary-increment harmonizable $S\alpha S$ motion $X=\{X(t):t\in\mathbb{R}\}$ defined by 
\begin{align}
	\label{statincrementSAS}
	X(t)=Re\left(\int_{\mathbb{R}}\left(e^{itx}-1\right)\Psi(x)M_\alpha(dx)\right),\quad t\in\mathbb{R},
\end{align}
where $M_\alpha$ is a complex isotropic $S\alpha S$ random measure with Lebesgue control measure $m(dx)=dx$ on $\mathbb{R}$. Here, the function $\Psi$ is a real valued function belonging to the weighted Lebesgue space $L^\alpha(\mathbb{R},\min(1,\vert x\vert^\alpha)dx)$ \cite{ayache2}.
Aside from the index of stability $\alpha$, the function $\Psi$ is the new estimation target, as it uniquely defines the process $X$. 

For both processes \eqref{newharmonizable} and \eqref{statincrementSAS}, the estimation of the spectral density $f$ or the function $\Psi$, respectively, is essentially the same problem as $\widetilde{M}_\alpha(dx)=\Psi(x)M_\alpha(dx)$ is a complex, isotropic and $S\alpha S$ with control measure $\widetilde{m}(dx)=\vert\Psi(x)\vert^\alpha m(dx)$, compare \cite[Proposition 3.5.5]{gennady}.

In some cases the function $\Psi$ in \eqref{statincrementSAS} is (assumed to be) explicitly known. For $\Psi(x)=\vert x\vert^{H-1/\alpha}$ with $H\in(0,1)$, the process
\begin{align*}
	X^H(t)=Re\left(\int_{\mathbb{R}}\left(e^{itx}-1\right)\vert x\vert^{-H-1/\alpha}M_\alpha(dx)\right),\quad t\in\mathbb{R},
\end{align*}
defines the so-called \emph{real harmonizable fractional stable motion} \cite[Section 7.7]{gennady}. 
In the Gaussian case $\alpha=2$, the process is the well-known fractional Brownian motion, which admits both a moving average as well as a harmonizable representation. 
The generalization of the harmonizable representation to the case $\alpha\in(0,2)$ is the aforementioned harmonizable fractional stable motion, whereas the generalization of its moving average representation to $\alpha\in(0,2)$ yields the class of linear fractional stable motions \cite[Section 7.4]{gennady}. 

The main goal here is to estimate the index of stability $\alpha$ and the parameter $H$ from a single path of $X^H$. 
Similar to the Gaussian case, we refer to the parameter $H$ as the \emph{Hurst parameter} of the process $X^H$.

\section{Mollifier}
\label{section: mollifier}

Let $X=\{X(t):t\in\mathbb{R}\}$ be a measurable stationary-increment harmonizable $S\alpha S$ motion with 
\begin{align*}
	X(t)=Re\left(\int_{\mathbb{R}}\left(e^{itx}-1\right)\Psi(x)M_\alpha(dx)\right),\quad t\in\mathbb{R},
\end{align*}
where $M_\alpha$ is a complex isotropic $S\alpha S$ random measure with Lebesgue control measure $m(dx)=dx$ and $\Psi$ is a real-valued function in the weighted Lebesgue space $L^\alpha(\mathbb{R},\min(1,\vert x\vert^\alpha)dx)$.
If $X(t)$ is almost surely integrable with respect to $v(t)dt$, we define the \emph{mollified process} $\tilde{X}_{v}=\{\tilde{X}_{v}(s):s\in\mathbb{R}\}$ by 
\begin{align}
	\label{mollified process}
	\tilde{X}_{v}(s)=\int_{\mathbb{R}}X(t)v\left(t-s\right)dt,\quad s\in\mathbb{R}. 
\end{align}
For the ease of notation, we write $\tilde{X}=\tilde{X}_{v}$ in the following.

The objective of this section is to analyze the distributional behavior of the mollified process $\tilde{X}$ as well as specifying conditions on $v$ for $\tilde{X}$ to be well-defined. 
Indeed, it can be shown that for a suitable choice of a smoothing kernel $\nu$ the smoothed process $\tilde{X}$ is a stationary real harmonizable stable process with finite control measure. 

First, we define the Fourier transform and the inverse Fourier transform on the space of integrable and square-integrable functions on $\mathbb{R}$ as follows. 
Let $v,w\in L^1(\mathbb{R})\cap L^2(\mathbb{R})$. Then, 
\begin{align*}
	\mathcal{F}v(x)=\int_\mathbb{R}e^{itx}v(t)dt,
	\quad
	\mathcal{F}^{-1}w(t)=\frac{1}{2\pi}\int_\mathbb{R}e^{-itx}w(x)dx.
\end{align*}
We denote by $W^{k,p}(\mathbb{R})$ the Sobolev space 
\begin{align*}
	W^{k,p}(\mathbb{R})=\left\{ u \in L^p(\mathbb{R}): u^{(j)}\in L^p(\mathbb{R}), \ j=1,\dots,k\right\}.
\end{align*}
It is well-known that the smoothness of a function and the rate of decay of its Fourier transform at the tails are directly related. In particular, it holds that $ t^k\mathcal{F}u(t)\rightarrow 0$ as $\vert t\vert\rightarrow\infty$ if $u\in W^{1,k}(\mathbb{R})$. 
This is a direct consequence of the fact that the Fourier transform of the $k$-th derivative $u^{(k)}$ of a function $u$ is given by $\mathcal{F}u^{(k)}(t)=i^kt^k\mathcal{F}u(t)$, which vanishes at its tails by the Riemann-Lebesgue lemma \cite[Propositions 2.2.11(9), 2.2.17]{grafakos}.

The following result states that under a suitable choice of the mollifier $v$, the mollified process $\tilde{X}_v$ is a stationary real harmonizable $S\alpha S$ with finite control measure. 
Its proof (given in \ref{appendix: proof theorem mollify}) hinges on integrability conditions for the paths of $\alpha$-stable processes given in \cite{integrability} and so-called \emph{$\alpha$-sine transforms}.
\begin{thm}
	\label{thm: mollify}
	Let $w\in L^1(\mathbb{R})\cap L^2(\mathbb{R})$ be an even function that vanishes at $0$ with $w(x)=O(\vert x\vert)$ as $\vert x\vert\rightarrow 0$. Additionally, assume that $w$ lies in the space $W^{1,3}(\mathbb{R})$. 
	Denote by $v=\mathcal{F}^{-1}w$ the Fourier inverse of $w$.
	Then, for a measurable stationary increment harmonizable $S\alpha S$ motion \eqref{statincrementSAS}, $\alpha\in(0,2)$, it holds that
		\begin{align}
			\label{integrabilitycondition}
			\int_{\mathbb{R}}\vert X(t)\vert v(t) dt <\infty\quad a.s.
		\end{align}
	Furthermore, the mollified process $\tilde{X}_v=\{\tilde{X}_v(s):s\in\mathbb{R}\}$ 
	is a stationary real harmonizable $S\alpha S$ process with finite control measure 
		$$\tilde{m}(dx)=\left\vert w\left(x\right) \Psi(x)\right\vert^\alpha dx.$$
\end{thm}

\begin{rem}
	\label{rem: scaling}
	One might introduce a tuning parameter $\theta>0$ which scales the smoothing function $v$ by $v(\theta \ \cdot)$ such that the mollified process $\tilde{X}_{v,\theta}=\{\tilde{X}_{v,\theta}(s):s\in\mathbb{R}\}$ is given by  
	\begin{align*}
		\tilde{X}_{v,\theta}(s) = \int_\mathbb{R} X(t)v\left(\theta(t-s)\right) dt. 
	\end{align*}
	By the scaling property of the Fourier transform, i.e. $\mathcal{F}v(\theta \ \cdot)(x) = \mathcal{F}v(x/\theta)/\theta=w(x/\theta)/\theta$, this directly translates to the control measure $\tilde{m}_\theta$ as 
	\begin{align*}
		\tilde{m}_\theta(dx)=\left\vert \frac{1}{\theta}w\left(\frac{x}{\theta}\right)\Psi(x)\right\vert^\alpha dx.
	\end{align*}
\end{rem}

\section{Statistical inference}
\label{section: statistical inference}

Throughout this section the following notation is used. 
Denote by $\tilde{X}_{\theta}=\tilde{X}_{v,\theta}$ the mollified process $\tilde{X}_\theta(s)=\int_\mathbb{R}X(t)v(\theta(t-s))dt$ with smoothing function $v$ and scaling parameter $\theta>0$. 
By Remark \ref{rem: scaling}, the control measure of $\tilde{X}_\theta$ is given by $\tilde{m}_\theta(dx)=\tilde{f}_\theta(x)dx$ with the spectral density $$\tilde{f}_\theta(x)=\left\vert\frac 1\theta w\left(\frac x\theta \right)\Psi(x)\right\vert^{\alpha}\in L^{1}(\mathbb{R}).$$
Let $\rho_\theta$ be the probability density function of the underlying random frequencies $\{Z_k\}_{k=1}^\infty$ in the process' series representation with
\begin{align*}
	\rho_\theta=\tilde{f}_\theta\left(\int_\mathbb{R}\tilde{f}_\theta(x)dx\right)^{-1}.
\end{align*} 
Choose $w\in L^1(\mathbb{R})\cap L^2(\mathbb{R})$ to be a non-negative, even and continuous function on $\mathbb{R}$, and denote $w_\theta(\cdot) = w(\cdot/\theta)/\theta$. It follows that 
\begin{align}
	\label{eq: rho theta}
	\rho_\theta= w_\theta^{\alpha}\vert\Psi\vert^\alpha c_{\alpha,\theta}\,,
\end{align}
where $c_{\alpha,\theta}=\left(\int_\mathbb{R}w_\theta^{\alpha}(x)\Psi^\alpha(x)dx\right)^{-1} $.
Additionally, assume that the function $\Psi$ is even and continuous. 

\subsection{Periodogram frequency estimation for stationary real harmonizable $S\alpha S$ processes}
\label{subsection: periodogram}

Any stationary real harmonizable $S\alpha S$ process $\tilde{X}=\{\tilde{X}(t):t\in\mathbb{R}\}$ with control measure $\tilde{m}(dx)=\tilde{f}(x)dx$ (and spectral density $\tilde{f}$) admits the LePage-type series representation 
\begin{align}
	\label{SRH-LePage}
	\tilde{X}(t)\stackrel{d}{=}C_\alpha \tilde{m}(\mathbb{R})^{1/\alpha} \sum_{k=1}^\infty\Gamma_k^{-1/\alpha}\left(G_k^{(1)}\cos(tZ_k)+G_k^{(2)}\sin(tZ_k)\right),\quad t\in\mathbb{R},
\end{align}
where $\{\Gamma_k\}_{k\in\mathbb{N}}$ are the arrival times of a unit rate Poisson process, $\{G^{(j)}_k\}_{k\in\mathbb{N}}$, $j=1,2$, are sequences of i.i.d. standard normal random variables, $\{Z_k\}_{k\in\mathbb{N}}$,
is a sequence of i.i.d. random variables with probability density function $\rho(\cdot) = \tilde{f}(\cdot)/\int_\mathbb{R}\tilde{f}(x)dx$, and  
\begin{align*}
	C_\alpha=\left(2^{\alpha/2}\Gamma\left(1+\frac{\alpha}{2}\right)\int_0^\infty \frac{\sin(x)}{x^{\alpha}}dx\right)^{-1/\alpha}.
\end{align*}
All sequences of random variables above are independent of each other.
In particular,  conditionally on $\{\Gamma_k\}_{k=1}^\infty$, $\{Z_k\}_{k=1}^\infty$, the process $\tilde{X}$ is Gaussian with purely atomic spectral measure concentrated at the points $\{\pm\vert Z_k\vert\}_{k=1}^\infty$, cf \cite[Proposition 6.6.4]{gennady}, which leads to the following approach.

A well-known estimation tool in spectral analysis of stationary processes is the so-called periodogram. 
The \emph{periodogram} $I(\theta)$ of the sample $(x(1),\dots,x(n))=(\tilde{X}(t_1),\dots,\tilde{X}(t_n))$ with equidistant sample points $t_k=k\delta$, $k=1,\dots,n$ and $\delta>0$, from a single path of $\tilde{X}$ as in \eqref{SRH-LePage} is defined by 
\begin{align*}
	I_n\left(\theta_j\right)= n^{-2}\left\vert \sum\limits_{k=1}^nx(k)e^{-ik\theta_j}\right\vert^2,
\end{align*}
where $\{\theta_j\}_{j=1}^n=\left\{2\pi j/n\right\}_{j=1}^n$ are the so-called Fourier frequencies. 
Straightforward computation yields that for any fixed $k\in\mathbb{N}$ the periodogram behaves like
\begin{align*}
	I_n(\delta\theta)
	=\frac{R_k^2}{4}\frac{\sin^2\left(\frac{n}{2}\delta(\vert Z_k\vert -\theta)\right)}{n^2\sin^2\left(\frac{1}{2}\delta(\vert Z_k\vert -\theta)\right)}
	+ O\left(n^{-2}\right)
\end{align*}
as $\theta\rightarrow \vert Z_k\vert$ and $n\rightarrow\infty$, where 
\begin{align}
	\label{amplitudes R}
	R_k&
	=C_\alpha\tilde{m}(\mathbb{R})^{1/\alpha}\Gamma_k^{-1/\alpha}\sqrt{\left(G_k^{(1)}\right)^2+\left(G_k^{(2)}\right)^2}
\end{align}
are referred to as amplitudes and $Z_k$ as frequencies of $\tilde{X}$, see \cite[Theorem 2]{viet2}.
It follows that the periodogram shows pronounced peaks at the frequencies close to the absolute frequencies $\vert Z_k\vert$.

Assuming that the spectral density $\rho$ is an even probability density function on $\mathbb{R}$, i.e. $\rho(x-)=\rho(x)$ for all $x\in\mathbb{R}$, the random frequencies $Z_k$ are i.i.d. symmetric random variables with probability density function $\rho$, in particular the absolute frequencies $\vert Z_k\vert$ are i.i.d. with probability density function $2\rho(x)$, $x\geq 0$.

Let $\{\hat{Z}_{N,n}\}_{k=1}^N$ be the locations of the $N$ largest peaks of the periodogram $I_n$ (computed as in \cite[Section 4]{viet2}). 
In addition, let $R_{[k]}$ be the amplitudes in descending order, i.e. $R_{[1]}>R_{[2]}>\dots$ a.s., and denote by $Z_{[k]}$ the corresponding frequencies. 
It can be shown that $\{\hat{Z}_{N,n}\}_{k=1}^N$ are strongly consistent estimators for $\{\vert Z_{[k]}\vert\}_{k=1}^N$ \cite[Theorem 4]{viet2}.
Ultimately, the kernel density estimator  
\begin{align}
	\label{kdeZhat}
	\hat \rho^{(N,n)}(x)=\frac{1}{2Nh_N}\sum\limits_{k=1}^N\kappa\left(\frac{x-\hat Z_{k,n}}{h_N}\right)
\end{align}
with kernel function $\kappa$ and bandwidth $h_N$ yields an estimate for the spectral density $f$ on the positive real line.
Under minor assumptions on $\rho$, $\kappa$ and $\{h_N\}_{N\in\mathbb{N}}$ the pointwise and uniform consistency of $\hat{\rho}_{N,n}$ in the weak and strong sense can be established \cite[Corollary 2]{viet2}.
\begin{lem}
	\label{lem: kdeconsistency}
	Let $\kappa$ be a Lipschitz-continuous kernel function and $h_N\rightarrow0$ with $Nh_N\rightarrow\infty$. Then, for any $\varepsilon>0$ it holds that
	\begin{align*}
		\lim_{N\rightarrow\infty}\lim_{n\rightarrow\infty} \mathbb{P}\left(\left\vert\hat\rho^{(N,n)}(z)-\rho(z)\right\vert>\varepsilon\right) = 0
	\end{align*}
	for every continuity point $z\in\mathbb{R}$ of $\rho$.
\end{lem}

\begin{rem}
	\label{rem: Rd periodogram}
	All processes introduced in Section \ref{subsection: problem setting} can be extended to random fields on $\mathbb{R}^d$, $d\geq 2$ by setting $(E,\mathcal{E})=(\mathbb{R}^d,\mathcal{B}(\mathbb{R}^d))$ and replacing all products $tx$, $t,x\in\mathbb{R}$ by scalar products $(\bm{t},\bm{x})$, $\bm{t},\bm{x}\in\mathbb{R}^d$,
	in the definitions of Section \ref{section: prelim} .
	The periodogram of a sample from  a stationary real harmonizable $S\alpha S$ random field shows peaks corresponding to the frequencies $\bm{Z}_k=(Z_{k,1},\dots,Z_{k,d})$ in all orthants of the Euclidean space $\mathbb{R}^d$ at $(\epsilon_1Z_{k,1},\dots,\epsilon_2Z_{k,d})$, where $\epsilon_k=\pm 1$ for all $k=1,\dots,d$. 
	Hence, analogous to the symmetry condition in the one-dimensional case, the spectral density $\Psi:\mathbb{R}^d\rightarrow\mathbb{R}$ must satisfy 
	$\Psi(x_1,\dots,x_d)=\Psi(\epsilon_1x_1,\dots,\epsilon_dx_d)$ for all $x\in\mathbb{R}^d$ and $\epsilon_k=\pm 1$, $k=1,\dots,d$.
	In practice, computational expenses for the multidimensional discrete Fourier transform, on which the periodogram relies, increase drastically with higher dimensions. For this reason, we only consider the case $d=1$ in this work. 
\end{rem}

\subsection{Stationary-increment harmonizable $S\alpha S$ motions}
\label{section: stat incr proc}

We consider the stationary-increment harmonizable $S\alpha S$ motion 
\begin{align*}
	X(t)=Re\left(\int_\mathbb{R}\left(e^{itx}-1\right)\Psi(x)M_\alpha(dx)\right),\quad t\in\mathbb{R},
\end{align*}
where $\Psi\in L^\alpha(\mathbb{R},\min(1,\vert x\vert^\alpha)dx)$ and $M_\alpha$ is a complex, isotropic $S\alpha S$ random measure with Lebesgue control measure $m(dx)=dx$.
Let $\theta_1,\theta_2>0$. Note that for all $z\in\mathbb{R}$ for which the ratio $\rho_{\theta_1}(z)/\rho_{\theta_2}(z)$ is well defined, it is also independent of $\Psi(x)$, in particular, by \eqref{eq: rho theta}
\begin{align}
	\label{ratio}
	\frac{\rho_{\theta_1}(z)}{\rho_{\theta_2}(z)} = \frac{c_{\alpha,\theta_1}w_{\theta_1}^\alpha(z)}{c_{\alpha,\theta_2}w_{\theta_2}^\alpha(z)}.
\end{align}
Moreover, if $w$ is an even, continuous function then so is $\rho_{\theta_1}/\rho_{\theta_2}$, hence it suffices to consider the positive real line only. 

Consider $z_1,\dots,z_L>0$, $L\in\mathbb{N}$. 
Applying a logarithmic transformation to Equation \eqref{ratio} yields 
\begin{align}
	\label{linearregression}
	\mathbf{y}_l = \log\left(\frac{\rho_{\theta_1}(z_l)}{\rho_{\theta_2}(z_l)}\right) = \log\left(\frac{c_{\alpha,\theta_1}}{c_{\alpha,\theta_2}}\right)+\alpha\log\left(\frac{w_{\theta_1}(z_l)}{w_{\theta_2}(z_l)}\right) = b+\alpha \mathbf{x}_l, \quad l=1,\dots,L.
\end{align}

Let $(x_{\theta_i,1},\dots,x_{\theta_i,n})=(\tilde{X}_{\theta_i}(s_1),\dots,\tilde{X}_{\theta_i}(s_n))$, $i=1,2$, be samples of the mollified processes $\tilde{X}_{\theta_i},$ $i=1,2$, sampled at equidistant points $s_j=j\delta$, $j=1,\dots,n$, $n\in\mathbb{N}$, and $\delta>0$. 
Similar to Equation \eqref{kdeZhat}, the kernel density type estimator $\hat\rho_{\theta_i}$ for $\rho_{\theta_i}$ is given by 
\begin{align*}
	\hat\rho_{\theta_i}^{(N,n)}(z) = \frac{1}{2Nh_N}\sum_{k=1}^N\kappa\left(\frac{z-\hat{Z}_{\theta_i,k,n}}{h_N}\right),\quad z\in\mathbb{R},
\end{align*} 
where $\{\hat{Z}_{\theta_i,k,n}\}_{k=1}^N$, $N\in\mathbb{N}$, $i=1,2$, are the periodogram frequency estimators attained from the sample of $\tilde{X}_{\theta_i}$ as described above. 

Define target vectors $\mathbf{Y}_L,\mathbf{\hat{Y}}_L^{(N,n)}\in\mathbb{R}^L$ and the design matrix $\mathbf{X}_L\in\mathbb{R}^{L\times 2}$ by
\begin{align*}
	\mathbf{Y}_L=\left(
		\begin{matrix}
			\log\left(\rho_{\theta_1}(z_1)\right)-\log\left(\rho_{\theta_2}(z_1)\right)\\
			\vdots\\
			\log\left(\rho_{\theta_1}(z_L)\right)-\log\left(\rho_{\theta_2}(z_L)\right)
		\end{matrix}
	\right),
\quad
	\mathbf{\hat{Y}}_L^{(N,n)}=\left(\begin{matrix}
		\log\left(\hat\rho^{(N,n)}_{\theta_1}(z_1)\right)-\log\left(\hat\rho^{(N,n)}_{\theta_2}(z_1)\right)\\
		\vdots\\
		\log\left(\hat\rho^{(N,n)}_{\theta_1}(z_L)\right)-\log\left(\hat\rho^{(N,n)}_{\theta_2}(z_L)\right)
	\end{matrix}\right),
\end{align*}
and
\begin{align*}
	\mathbf{X}_L=
	\left(
		\begin{matrix}
			1&\log\left(w_{\theta_1}(z_1)\right)-\log\left(w_{\theta_2}(z_1)\right)\\
			\vdots&\vdots\\
			1&\log\left(w_{\theta_1}(z_L)\right)-\log\left(w_{\theta_2}(z_L)\right)
		\end{matrix}
	\right).
\end{align*}
Then, the system of linear equations \eqref{linearregression} generates a linear regression model written in matrix form as
$
	\mathbf{Y}_L=\mathbf{X}_L\bm{\beta}+\bm{\varepsilonup}_L
$, 
where $\bm{\beta}=(b,\alpha)^T$ and $\bm{\varepsilonup}_L$ is a vector of possible computation errors with mean vector $\mathbb{E}[\bm\varepsilonup_L] = \bm{0}$ and covariance matrix $\bm{Q}_L=\text{Cov}(\bm\varepsilonup_L)$. 

\begin{rem}
	\label{rem: consistency LSE}
	Consistency of the least-squares estimator $\hat{\bm\beta}_L=(\bm{X}_L^T\bm{X}_L)^{-1}\bm{X}_L^T\bm{Y}_L$ in a linear regression model $\bm{Y}_L=\bm{X}_L\bm{\beta}+\bm{\varepsilonup_L}$ is guaranteed under the following conditions \cite{drygas}. Denote by $\lambda_{\min}(\bm{A})$, $\lambda_{\max}(\bm{A})$ the smallest and largest eigenvalues of a square matrix $\bm{A}$.
	
	If the boundary conditions $\lambda_{\max}(Q_L)<\infty$, $\lambda_{\min}(Q_L)>0$ are fulfilled and $\lambda_{\min}(\bm{X}^T\bm{X})\rightarrow\infty$ as $L\rightarrow\infty$, then the least-squares estimator $\hat{\bm\beta}_L$ of $\bm\beta$ is weakly consistent. 
	The above is fulfilled for the usual assumptions on linear regression models, i.e. when the vector of errors $\bm{\varepsilonup}$ is homoscedastic with uncorrelated components, and $L^{-1}X_L^TX_L$ converges to some covariance matrix $Q$ as $L\rightarrow\infty$.
	Additionally, conditions for strong consistency are given in \cite[Section 4]{drygas}.
\end{rem}

Consider $z_1,\dots,z_L$, $L\in\mathbb{N}$, such that $w_{\theta_1}(z_i)/w_{\theta_2}(z_i)\neq w_{\theta_1}(z_j)/w_{\theta_2}(z_j)$ for some $i\neq j$.
This ensures that the matrix $\bm{X}^T\bm{X}$ is invertible and its eigenvalues $\lambda_{1},\lambda_2$ can be computed explicitly, i.e. 
\begin{align*}
	\lambda_{1,2} = \frac{1}{2}\left(L+\sum_{l=1}^Lx_l^2\pm\sqrt{\left(L+\sum_{l=1}^Lx_l^2\right)^2-4\left(L\sum_{l=1}^Lx_l^2-\left(\sum_{l=1}^Lx_l\right)^2\right)}\right),
\end{align*}
where $x_l = \log(w_{\theta_1}(z_l)/w_{\theta_2}(z_l))$, $l=1,\dots,L$. 
It is easy to verify that $\lambda_{\min} (\bm{X}^T\bm{X})\rightarrow\infty$ as $L\rightarrow\infty$.  

Denote by $\bm{\beta}^{(N,n)}_L=(\hat{b}_L^{(N,n)},\hat\alpha_L^{(N,n)})^T$ the least-squares estimator for $\bm\beta$ in the linear regression model $\hat{\bm{Y}}_L^{(N,n)}=\bm{X}_L\bm\beta+\bm{\varepsilonup_L}$, i.e.
\begin{align*}
	\bm{\beta}^{(N,n)}_L=\left(\hat{b}_L^{(N,n)}, \hat\alpha_L^{(N,n)}\right)^T =\left( \mathbf{X}_L^T\mathbf{X}_L\right)^{-1}\mathbf{X}_L^T\mathbf{\hat{Y}}_L^{(N,n)}.
\end{align*}

\begin{cor}
	\label{cor: consistent alpha}
	Let the kernel density estimators $\hat\rho_{\theta_i}^{(N,n)}$, $i=1,2$, satisfy the conditions of Lemma \ref{lem: kdeconsistency}. Assume that the conditions on the error vector $\bm\varepsilonup_L$ in Remark \ref{rem: consistency LSE} are fulfilled. 
	Then, the least-squares estimator $\bm\beta_L^{(N,n)}$ is consistent for $\bm\beta$ in the sense that 
	\begin{align}
		\label{eq: triple limit consistency}
		\lim_{L\rightarrow\infty}\lim_{N\rightarrow\infty}\lim_{n\rightarrow\infty} \mathbb{P}\left(\left\vert\hat{b}_L^{(N,n)}-b\right\vert>\varepsilon\right) = 0, \quad
		\lim_{L\rightarrow\infty} \lim_{N\rightarrow\infty}\lim_{n\rightarrow\infty} \mathbb{P}\left(\left\vert\hat{\alpha}_L^{(N,n)}-\alpha\right\vert>\varepsilon\right) = 0.
	\end{align}
\end{cor}
\begin{proof}
	The full rank of $\mathbf{X}^T\mathbf{X}$ is ensured by the choice of $z_1,\dots,z_L$, thus the linear least-squares  minimization has a unique solution $(b,\alpha)^T$. Note that in the following when referring to consistency of $\hat\rho_{\theta}^{(N,n)}$, we always consider the double limits in the order $\lim_{N\rightarrow\infty}\lim_{n\rightarrow\infty}$. 
	Under the conditions of Lemma \ref{lem: kdeconsistency} the estimators $\hat\rho_{\theta_i}^{(N,n)}$, $i=1,2$, are 
	weakly pointwise consistent at $z_1,\dots,z_L$. Hence, it follows that $\mathbf{\hat{Y}}^{(N,n)}$ is weakly consistent for $\mathbf{Y}$ by the continuous mapping theorem, and
	$$
	\hat{\bm\beta}^{(N,n)}_L=\left( \mathbf{X}_L^T\mathbf{X}_L\right)^{-1}\mathbf{X}_L^T\mathbf{\hat{Y}}_L^{(N,n)}\overset{P}{\longrightarrow}
	\left(\mathbf{X}_L^T\mathbf{X}_L\right)^{-1}\mathbf{X}_L^T\mathbf{\hat{Y}}_L=:\hat{\bm\beta}_L
	$$
	as first $n\rightarrow\infty$, then $N\rightarrow\infty$. 
	By the triangle inequality 
	$$\left\vert \hat{\bm\beta}^{(N,n)}_L-\bm\beta\right\vert\leq\left\vert\hat{\bm\beta}^{(N,n)}_L-\hat{\bm\beta}_L\right\vert+\left\vert\hat{\bm\beta}_L-\bm\beta\right\vert$$
	Equations \eqref{eq: triple limit consistency} follows.
\end{proof}

\begin{rem}
	\label{rem: kdeconsistency}
	If $\rho_\theta$ is known to be uniformly continuous and $Nh_N^2\rightarrow\infty$ holds, then 
	\begin{align*}
		\lim_{N\rightarrow\infty}\lim_{n\rightarrow\infty} \mathbb{P}\left(\left\Vert\hat\rho^{(N,n)}_{\theta}-\rho_\theta\right\Vert_\infty>\varepsilon\right) = 0.
	\end{align*}
	Furthermore, strong uniform consistency 
	\begin{align*}
		\mathbb{P}\left(\lim_{N\rightarrow\infty}\lim_{n\rightarrow\infty} \left\Vert\hat\rho^{(N,n)}_{\theta}-\rho_\theta\right\Vert_\infty=0\right) = 1
	\end{align*}
	is ensured if, additionally to its Lipschitz-continuity as in Lemma \ref{lem: kdeconsistency}, the kernel function $\kappa$ has bounded variation, vanishes at $\pm\infty$ and the bandwidth satisfies $\sum_{N=1}^\infty\exp\left(-rNh_N^2\right)<\infty$ for any $r>0$, cf. \cite{viet2}
	By the same arguments as above this yields strong consistency of $(\hat{b}^{(N,n)},\hat\alpha^{(N,n)})^T$ for $(b,\alpha)^T$.
\end{rem}

Note that, a consistent estimator for $c_{\alpha,\theta}\Psi^\alpha$, i.e. an estimator for the function $\Psi^\alpha$ up to a constant factor, is given by $(\hat{\rho}_\theta^{(N,n)})w_\theta^{-\hat\alpha^{(N,n)}}$.
Pointwise and uniform consistency in the weak and strong sense can be accomplished by Corollary \ref{cor: consistent alpha} and Remark \ref{rem: kdeconsistency} under the respective conditions. 

\begin{rem}
	\label{rem: psi}
	The kernel function $\Psi$ can only be restored from the random frequencies $\{Z_k\}_{k=1}^\infty$ up to a constant multiplicative factor. Consider the scaled version $X_c$, $c>0$, of a stationary-increment harmonizable stable process $X$ with
	\begin{align*}
		X_c(t)=Re\left(\int_\mathbb{R}\left(e^{itx}-1\right)c\Psi(x)M_\alpha(dx)\right),\quad t\in\mathbb{R}.
	\end{align*}
	Then, the mollified process $\tilde{X}_{c,\theta}(s)=\int_\mathbb{R}X_c(t)v(\theta(t-s))dt$ is a stationary real harmonizable $S\alpha S$ process with finite control measure $\tilde{m}_{c,\theta}(dx)=c^\alpha w_\theta^\alpha(x)\vert \Psi(x)\vert^\alpha dx$. The underlying i.i.d. random frequencies $\{Z_k\}_{k=1}^\infty$ in the corresponding LePage series representation, upon which all estimates rely, are independent of the constant $c$, since it cancels in their probability density function 
	$$\rho_{c,\theta}(x)=c^\alpha w_\theta^\alpha(x)\vert\Psi(x)\vert^\alpha\left(\int_\mathbb{R}c^\alpha w_\theta^\alpha\vert\Psi(x)\vert^\alpha dx\right)^{-1}= w_\theta^\alpha(x)\vert\Psi(x)\vert^\alpha\left(\int_\mathbb{R} w_\theta^\alpha\vert\Psi(x)\vert^\alpha dx\right)^{-1}.$$
\end{rem}

\subsection{Real harmonizable fractional stable motions}
\label{section: fractionalmotion}

We apply the results of Section \ref{section: mollifier} to the class of {real harmonizable fractional stable motions} $X^H=\{X^H(t):t\in\mathbb{R}\}$ given by 
\begin{align*}
	X^H(t)=\mathrm{Re}\left(\int_\mathbb{R}\left(e^{itx}-1\right)\vert x\vert^{-H-1/\alpha}M_\alpha(dx)\right),
\end{align*}
where $\alpha\in(0,2)$, $H\in(0,1)$ and $M_\alpha$ is a complex isotropic $S\alpha S$ random measure with Lebesgue control measure $m(dx)=dx$. 
The additional pre-knowledge of $\Psi(x)=\vert x\vert^{-H-1/\alpha}$ allows us to propose an estimation procedure, which is explicitly tailored to real harmonizable fractional stable motions. 

\begin{cor}
	\label{corollary hfsm mollified}
	Let $w,v\in L^1(\mathbb{R})\cap L^2(\mathbb{R})$ be even functions given by $w(x)=\vert x\vert^pe^{-\vert x\vert^q}$, $p,q\geq 1$, and $v=\mathcal{F}^{-1}w$.
	\begin{enumerate}[(i)]
		\item The mollified process $\tilde{X}_\theta=\{\tilde{X}_\theta(s):s\in\mathbb{R}\}$ with $\tilde{X}_\theta(s)=\int_\mathbb{R} X^H(t)v(\theta(t-s))dt$, $s\in\mathbb{R}$, $\theta>0$, is a stationary real harmonizable $S\alpha S$ process with finite control measure 
		$\tilde{m}_\theta$.
		The random frequencies $\{Z_k\}_{k\in\mathbb{N}}$ in its LePage series representation \eqref{SRH-LePage} are i.i.d. with probability density function
		\begin{align}
			\label{rho_hfsm}
			\rho_\theta(z)=\frac{q}{2}\frac{\left(\alpha\theta^{-q}\right)^{\frac{\alpha}{q}(p-H)}}{\Gamma\left(\frac{\alpha}{q}(p-H)\right)}e^{-\alpha\theta^{-q}\vert x\vert^q}\vert x\vert^{\alpha (p-H)-1},\quad z\in\mathbb{R}.
		\end{align}
		\item \label{part3 gamma}It holds that $\{\vert Z_k\vert^q\}_{k\in\mathbb{N}}$ is a sequence of  i.i.d. Gamma-distributed random variables with shape parameter $\frac{\alpha}{q}(p-H)$ and scale parameter $\alpha\theta^{-q}$, i.e. 
		\begin{align*}
			\left\vert Z_k\right\vert^q\sim\Gamma\left(\frac{\alpha}{q}\left(p-H\right),\alpha\theta^{-q}\right).
		\end{align*}
	\end{enumerate}
\end{cor}

\begin{proof}
	\begin{enumerate}[(i)]
		\item \label{part i}The weight function $w$ fulfills all necessary conditions of Theorem \ref{thm: mollify}.
		It follows that $\int_{\mathbb{R}}\vert X^H(t)\vert v(t)dt<\infty$ a.s. and in particular the mollified process $\tilde{X}=\tilde{X}_{v,\theta}$ with
		\begin{align*}
			\tilde{X}(s)\stackrel{a.s.}{=}\int_{\mathbb{R}}e^{isx}w_\theta(x)\left\vert x\right\vert^{-H-1/\alpha}M_\alpha(dx),\quad s\in\mathbb{R},
		\end{align*}
		is a stationary real harmonizable $S\alpha S$ process with finite control measure 
		\begin{align*}
			\tilde{m}_\theta(dx)=\tilde{f}_\theta(x)dx=w_\theta^\alpha(x)\vert x\vert^{-\alpha H - 1} dx.
		\end{align*}
		Therefore,  it admits the LePage series representation \eqref{SRH-LePage}
		with frequencies $\{Z_k\}_{k\in\mathbb{N}}$, which are i.i.d. with probability density function $\rho_\theta=\tilde{f}_\theta\left(\int_\mathbb{R}\tilde{f}_\theta(x)dx\right)^{-1}$.
		
		Note that the spectral density $\tilde{f}$ is given by
		\begin{align*}
			\tilde{f}_\theta(x)=w_\theta^\alpha(x)\vert x\vert^{-\alpha H - 1}=\theta^{-\alpha(1+p)} e^{-\alpha\theta^{-q}\vert x\vert^q}\vert x\vert^{\alpha (p-H)-1}
		\end{align*}
		and 
		\begin{align*}
			\int_{\mathbb{R}}\tilde{f}_\theta(x) dx&=\theta^{-\alpha(1+p)} \int_{\mathbb{R}} e^{-\alpha\theta^{-q}\vert x\vert^q}\vert x\vert^{q\left(\frac{\alpha p}{q}-\frac{\alpha H}{q}-\frac{1}{q}\right)}dx\\
			&=\frac{2}{q}\theta^{-\alpha(1+p)}\frac{\Gamma\left(\frac{\alpha}{q}(p-H)\right)}{\left(\alpha\theta^{-q}\right)^{\frac{\alpha}{q}(p-H)}}
			\underbrace{\int_{0}^\infty \frac{\left(\alpha\theta^{-q}\right)^{\frac{\alpha}{q}(p-H)}}{\Gamma\left(\frac{\alpha}{q}(p-H)\right)} e^{-\alpha\theta^{-q}\eta}\eta^{\alpha-\frac{\alpha H}{q}-1}d\eta}_{=1},
		\end{align*}
		where we substituted $\eta = x^q$ in the above integration and expanded by the normalizing constant of the Gamma distribution with shape parameter $\frac{\alpha}{q}(p-H)$ and scale parameter $\alpha\theta^{-q}$.
		Consequently, this yields
		\begin{align*}
			\rho_\theta(z)=\frac{q}{2}\frac{\left(\alpha\theta^{-q}\right)^{\frac{\alpha}{q}(p-H)}}{\Gamma\left(\frac{\alpha}{q}(p-H)\right)}e^{-\alpha\theta^{-q}\vert z\vert^q}\vert z\vert^{\alpha (p-H)-1}.
		\end{align*}
	
		\item The sequence $\{\vert Z_k\vert^q\}_{k\in\mathbb{N}}$ consists of i.i.d. random variables. 
		Denote their probability density function by $g$. 
		By the density transformation theorem with transform $\varphi(z)=\vert z\vert^q$ we can compute 
		\begin{align*}
			g(z)&=\left(\rho_\theta\left(- z^{\frac{1}{q}}\right)+\rho_\theta \left( z^{\frac{1}{q}}\right)\right)\frac{1}{q}z^{\frac{1}{q}-1}\\
			&=2\frac{q}{2}\frac{\left(\alpha\theta^{-q}\right)^{\frac{\alpha}{q}(p-H)}}{\Gamma\left(\frac{\alpha}{q}(p-H)\right)}e^{-\alpha\theta^{-q}z}z^{\frac{1}{q}\left(\alpha (p-H)-1\right)}\frac{1}{q}z^{\frac{1}{q}-1}\\
			&=\frac{\left(\alpha\theta^{-q}\right)^{\frac{\alpha}{q}(p-H)}}{\Gamma\left(\frac{\alpha}{q}(p-H)\right)}e^{-\alpha\theta^{-q}z}z^{\frac{\alpha}{q}(p-H)-1},
		\end{align*}
		which is the probability density function of the Gamma distribution with shape parameter $\frac{\alpha}{q}(p-H)$ and scale parameter $\alpha\theta^{-q}$.
	\end{enumerate}
\end{proof}

Applying the periodogram method described in Section \ref{subsection: periodogram} to the mollified process $\tilde{X}_\theta$ the parameters $\alpha$ and $H$ can be estimated as follows. 
Consider the sample $(x_{\theta,1},\dots,x_{\theta,n})=(X_\theta(t_1),\dots,X_\theta(t_n))$ where $t_j=j\delta$, $j=1,\dots,n$, $\delta>0$. 
Denote by $\{\hat{Z}_{k,n}\}_{k=1}^N$ the periodogram frequency estimators of $\{Z_{[k]}\}_{k=1}^N$ as described in Section \ref{subsection: periodogram}. 

Consider an i.i.d. random sample $X_1,\dots,X_n\sim\Gamma(b,r)$, $b,r>0$.
Strongly consistent closed form estimators for the parameters $(b,r)$ are given in \cite{gammadistr}, i.e. 
\begin{align*}
 	\tilde{b}(X_1,\dots,X_N)&= \frac{N\sum_{k=1}^NX_k}{N\sum_{k=1}^NX_k\log(X_k)-\sum_{k=1}^N\log(X_k)\sum_{k=1}^NX_k},\\
 	\tilde{r}(X_1,\dots,X_N)&=\frac{N^2}{N\sum_{k=1}^NX_k\log(X_k)-\sum_{k=1}^N\log(X_k)\sum_{k=1}^NX_k}.
\end{align*}
The index of stability $\alpha$ and the Hurst parameter $H$ can therefore be estimated by 
\begin{align*}
	\tilde{\alpha}^{(N,n)} &= \theta^q\tilde{r}\left(\hat{Z}^q_{1,n},\dots,\hat{Z}^q_{N,n}\right),\\
	\tilde{H}^{(N,n)} &= p-q\frac{\tilde{b}\left(\hat{Z}^q_{1,n},\dots,\hat{Z}^q_{N,n}\right)}{\tilde{\alpha}^{(N,n)}}.
\end{align*}
\begin{cor}
	\label{cor: hfsm consistent alpha and H}
	The estimators $(\tilde{\alpha}^{N,n},\tilde{H}^{N,n} )$ are strongly consistent for $(\alpha,H)$, i.e.
	\begin{align*}
		\mathbb{P}\left(\lim_{N\rightarrow\infty}\lim_{n\rightarrow\infty}\tilde{\alpha}^{N,n}=\alpha\right)=1,\quad 
		\mathbb{P}\left(\lim_{N\rightarrow\infty}\lim_{n\rightarrow\infty}\tilde{H}^{N,n}=H\right)=1
	\end{align*}
\end{cor}
\begin{proof}
	The strong consistency of the periodogram frequency estimators $\{\hat{Z}_{k,n}\}_{k=1}^N$ for $\{\vert Z_{[k]}\vert\}_{k=1}^N$ was proven in \cite[Theorem 4]{viet2}. Clearly, by the continuity of the mapping $z\mapsto \vert z\vert ^q$ it follows that $\{\hat{Z}^q_{k,n}\}_{k=1}^N$ are strongly consistent for $\{ \vert Z\vert^q_{[k]}\}_{k=1}^N$, which are i.i.d. Gamma distributed with shape parameter $\alpha-\frac{\alpha H}{q}$ and rate $\alpha\theta^{-q}$ by Corollary \ref{corollary hfsm mollified} (\ref{part3 gamma}). Solving for $\alpha$ and $H$ and applying the continuous mapping theorem once again concludes the proof. 
\end{proof}
The choice of the parameters $p,q\geq 1$ is discussed in Remark \ref{rem: choice p q} in the next section.

\section{Simulation study}
\label{section: numerical}
In this section, we consider estimation based on a single path of a real harmonizable fractional stable motion ("real HFSM" in the following) $X^H$ as in Section \ref{section: fractionalmotion}.
Let $(x_1,\dots,x_n)=(X^H(t_1),\dots,X^H(t_n))$ be a sample of a single path of $X^H$, where the points $\{t_j\}_{j=1}^n$ are chosen equidistantly with $t_j=j\delta$, $j = 1,\dots,n$, and mesh size $\delta>0$.
The simulation is performed via the discretization of the integral representation of $X^H$. 
Realizations of $X^H$ for $\alpha\in\{0.75,1.5\}$ and $H\in\{0.25,0.75\}$ are given in Figure \ref{fig: rhfsm realizations} with the same simulation seed being used in each realization.
As it is the case for the classical fractional Brownian motion, the Hurst parameter $H$ determines the roughness of the paths. 
Additionally,  the smaller the index of stability $\alpha$ the higher the fluctuations as well.  
\begin{figure}
	\centering
	\begin{subfigure}{0.245\textwidth}
		\includegraphics[width=1.1\textwidth]{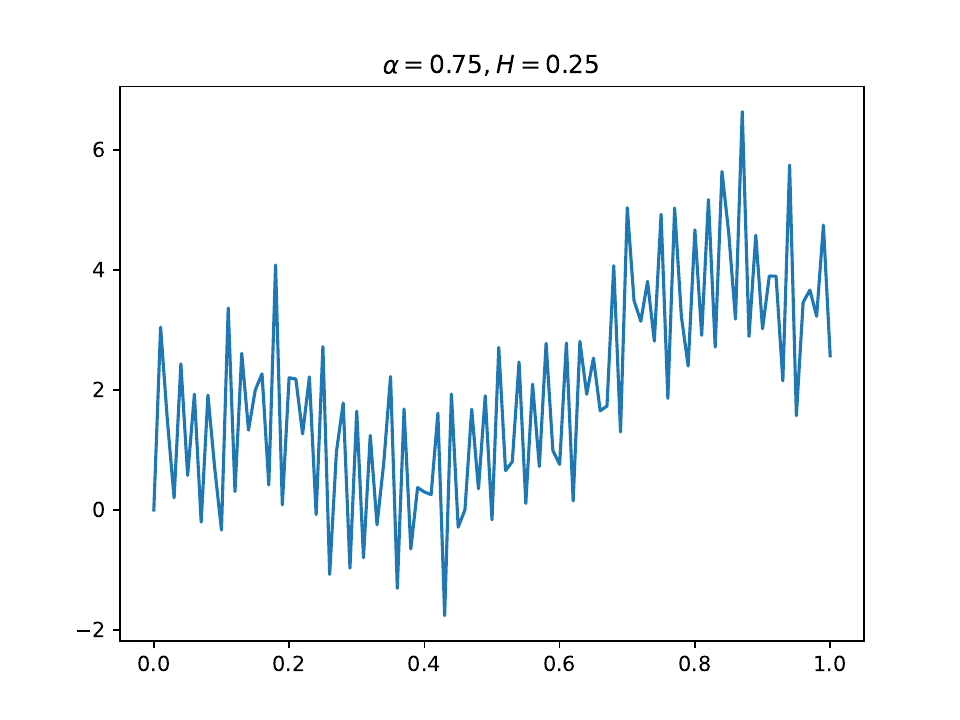}
		\caption{$\alpha=0.75,H=0.25$}
	\end{subfigure}
	\hfill
	\begin{subfigure}{0.245\textwidth}
		\includegraphics[width=1.1\textwidth]{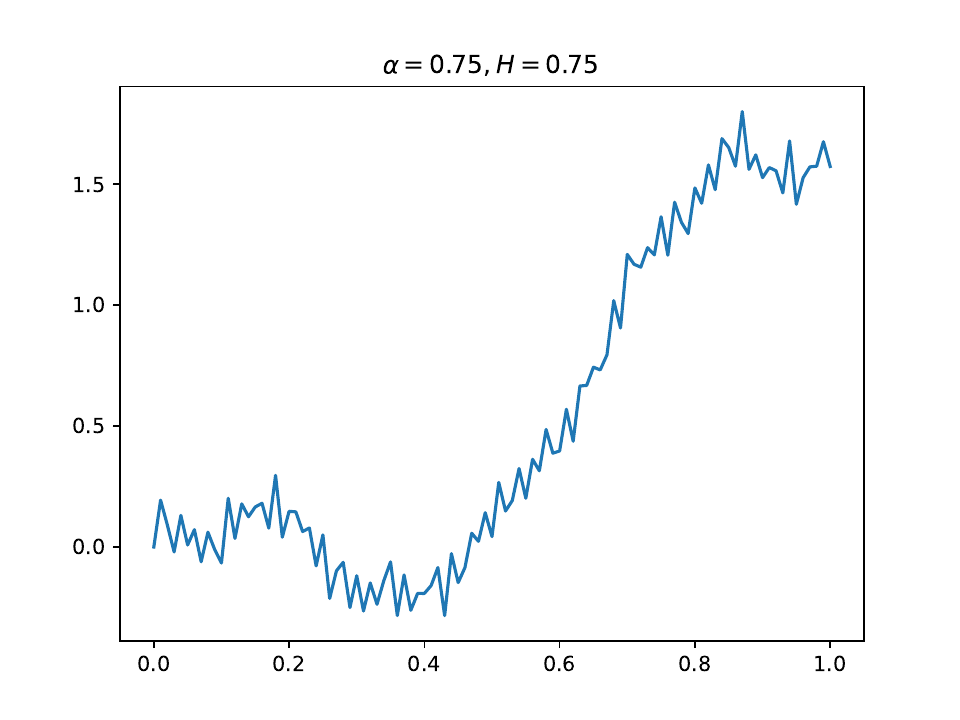}
		\caption{$\alpha=0.75,H=0.75$}
	\end{subfigure}
	\hfill
	\begin{subfigure}{0.245\textwidth}
		\includegraphics[width=1.1\textwidth]{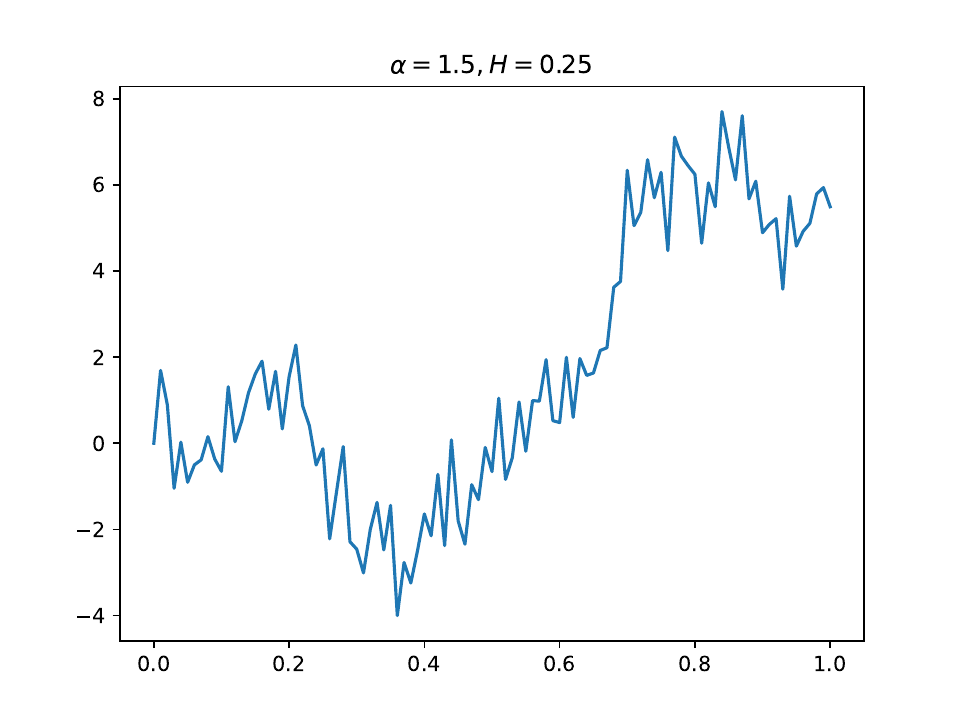}
		\caption{$\alpha=1.5,H=0.25$}
	\end{subfigure}
	\hfill
	\begin{subfigure}{0.245\textwidth}
		\includegraphics[width=1.1\textwidth]{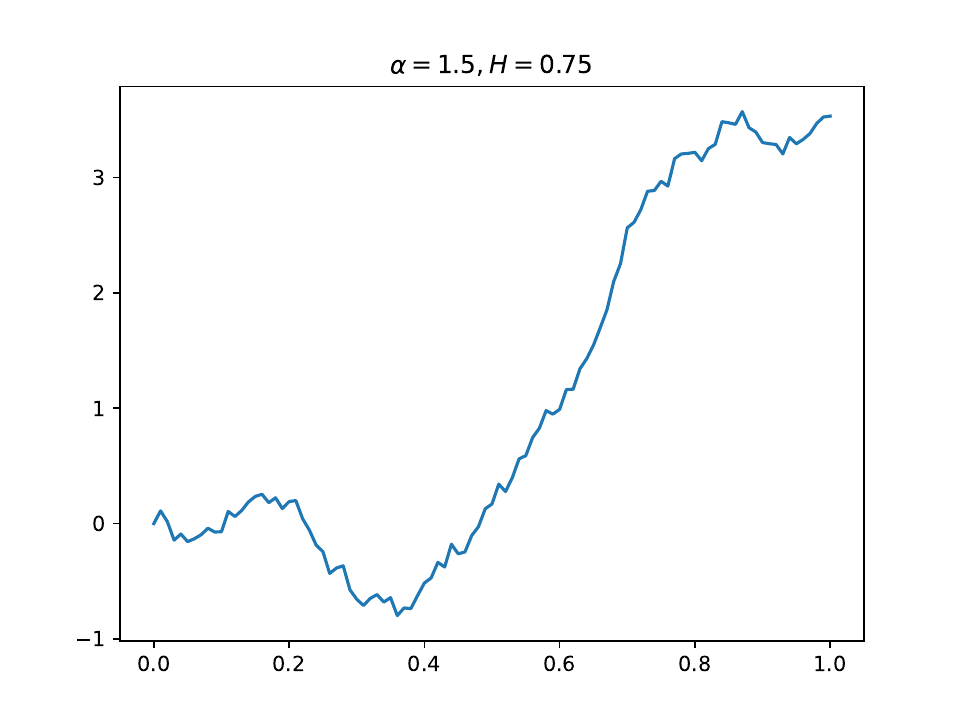}
		\caption{$\alpha=1.5,H=0.75$}
	\end{subfigure}
	\caption{Realization of a real HFSM for $\alpha\in\{0.75,1.5\}$ and $H\in\{0.25,0.75\}$ on $[0,1]$. The simulations are based on the same seed to emphasize the effect of $\alpha$ and $H$ on the paths' roughness.}
	\label{fig: rhfsm realizations}
\end{figure}

\begin{ex}
	\label{numex 1}
	Consider the even weight function $w(x)=\vert x\vert^pe^{-\vert x\vert^q}$, $p,q\geq 1$ of Corollary \ref{corollary hfsm mollified}. Setting the exponent $q=2$ in the exponential term of $w$ guarantees its (inverse) Fourier transform can be given explicitly. In particular, it holds that $v=\mathcal{F}^{-1}w$  with $w(x)=\vert x\vert^pe^{-x^2}$, $p\geq 1$, is of the form
	\begin{align*}
		v(t)=\frac{1}{2\pi}\Gamma\left(\frac{1+p}{2}\right){}_1F_1\left(\frac{1+p}{2};\frac{1}{2};-\frac{t^2}{4}\right),
	\end{align*} 
	where ${}_1F_1$ is a generalized hypergeometric series defined by ${}_1F_1(a;b;z)=\sum_{n=1}^\infty\frac{a^{(n)}z^n}{b^{(n)}n!}$, $z\in\mathbb{C}$, and rising factorials $a^{(0)}=1$, $a^{(n)}=a(a+1)\dots (a+n-1)$.
	
	In particular, for $p=2,4$ this yields
	\begin{align}
		\label{w1v1}
		w^{(1)}(x)=x^2e^{-x^2}, \quad v^{(1)}(t)=\frac{1}{8\sqrt\pi}e^{-\frac{t^2}{4}}\left(2-t^2\right)
	\end{align}
	and 
	\begin{align}
		\label{w2v2}
		w^{(2)}(x)=x^4e^{-x^2}, \quad v^{(2)}(t)=\frac{1}{32\sqrt\pi}e^{-\frac{t^2}{4}}\left(12-12t^2+t^4\right).
	\end{align}
	The choice of $p,q\geq 1$ and their effect on the asymptotic variance of the estimators $\tilde{b},\tilde{r}$ of the Gamma distribution are discussed in more detail in Remark \ref{rem: choice p q}.
\end{ex}

We compute samples of the path of the mollified process $\tilde{X}_{v,\theta}$ via discrete convolution. 
Fix an even number of discretization points $L\in\mathbb{N}$ of $v_\theta$.
Recall that the function $v_\theta$ is even, hence consider its discretization $(v_{\theta,0},\dots,v_{\theta,L})$ on the interval $[-\delta L/2,\delta L/2]$ with $v_l = v\left(\delta\left(-\frac{L}{2}+l\right)\right)$, $l=0,\dots,L$.
The convolution operation is performed only where the samples $(x_0,\dots,x_n)$ and $(v_{\theta,0},\dots,v_{\theta,L})$ fully overlap (in time), i.e. compute $(\tilde{x}_{\theta,0},\dots,\tilde{x}_{\theta,n-L})$ with 
\begin{align}
	\label{eq: discrete convolution}
	\tilde{x}_{\theta,k} = \delta \sum_{l=0}^Lx_{k-l}v_l\approx\tilde{X}_{v,\theta}\left(\delta\left(\frac{L}{2}+k\right)\right),\quad k=0,\dots,n-L.
\end{align}
For the numerical computation we use Python's \texttt{scipy.signal.convolve} function and its \texttt{mode} parameter set to \texttt{'valid'}.

The number $L$ should be chosen such that $n-L$ is still large enough for periodogram computations, i.e. $n-L\geq n_0$ for some desired minimal sample size $n_0$. 
Furthermore, the tuning parameter $\theta$ should be chosen such that the discretization $(v_{\theta,0},\dots,v_{\theta,L})$ on the points $\{\delta(-L/2+j)\}_{j=0}^L$ approximate $v_\theta$ well enough and $\vert v_\theta(t)\vert<\varepsilon$, $t\not\in[-\delta L/2,\delta L/2]$, for some small $\varepsilon>0$. 
For the discretization of $v_\theta$ we set $L=100$, $\delta=0.01$, which yields the interval $[-0.5,0.5]$ on which $v_\theta$ is discretized. 
Setting $\theta_0=20$ guarantees that $v_\theta$ is negligible outside of $[-0.5,0.5]$.

\subsection{Regression based estimation of $\alpha$}
Figure \ref{fig: regression type alpha estimation} visualizes the regression based estimation of the index of stability $\alpha$ as described in Section \ref{section: stat incr proc}. 
Subfigure (a) shows the true $\rho_{\theta_i}$ (dashed lines) against its estimates $\hat\rho_{\theta_i}$ (solid lines), $i=1,2$, where $\theta_1=20$, $\theta_2=30$. 
The estimation is based on the simulation of a single path of $X^H$ with $(\alpha,H)=(1.5,0.75)$ equidistant sample points with mesh size $\delta=0.01$ and $n=10000$. 
This path is then mollified twice with $v^{(1)}$ of Example \ref{numex 1}, once with $\theta_1$ and once with $\theta_2$.

Subfigure (b) shows the ratio $\hat\rho_{\theta_1}/\hat\rho_{\theta_2}$ against the true ratio $\rho_{\theta_1}/\rho_{\theta_2}$. 
Since it holds that $\rho_{\theta_1}/\rho_{\theta_2} = C (w_{\theta_1}/w_{\theta_2})^\alpha$, where $C$ is a constant and $w_{\theta_1}/w_{\theta_2}$ is known, it is possible to deduce the general shape of the true ratio $\rho_{\theta_1}/\rho_{\theta_2}$. 
Stronger deviations from this shape such as the downward kink of $\hat\rho_{\theta_1}/\hat\rho_{\theta_2}$ near the origin and the jump at its right tail can be explained by the fact that 
both $\rho_{\theta_1}$ and $\rho_{\theta_2}$ vanish at the origin at their tails, combined with inaccuracy of kernel density estimation at those ends. 

Before computing the log-transformed ratio of $\log(\hat\rho_{\theta_1}(z)/\hat\rho_{\theta_2}(z))$, we restrict ratio to the interval $[b_l,b_u]$, where the lower bound $b_l=1$ was chosen to exclude fluctuations near the origin and $b_u$ is chosen automatically such that there are no jumps at the right tail of the ratio and the ratio is larger than the threshold $\varepsilon = 0.05$. 
A visualization of the linear dependency in $\alpha$ (cf. Equation \eqref{linearregression}) can be seen in Subfigure (c). In this example, linear regression yields the estimate $\hat\alpha = 1.467$.

\begin{figure}[h]
	\centering
	\begin{subfigure}{0.325\textwidth}
		\includegraphics[width=\textwidth]{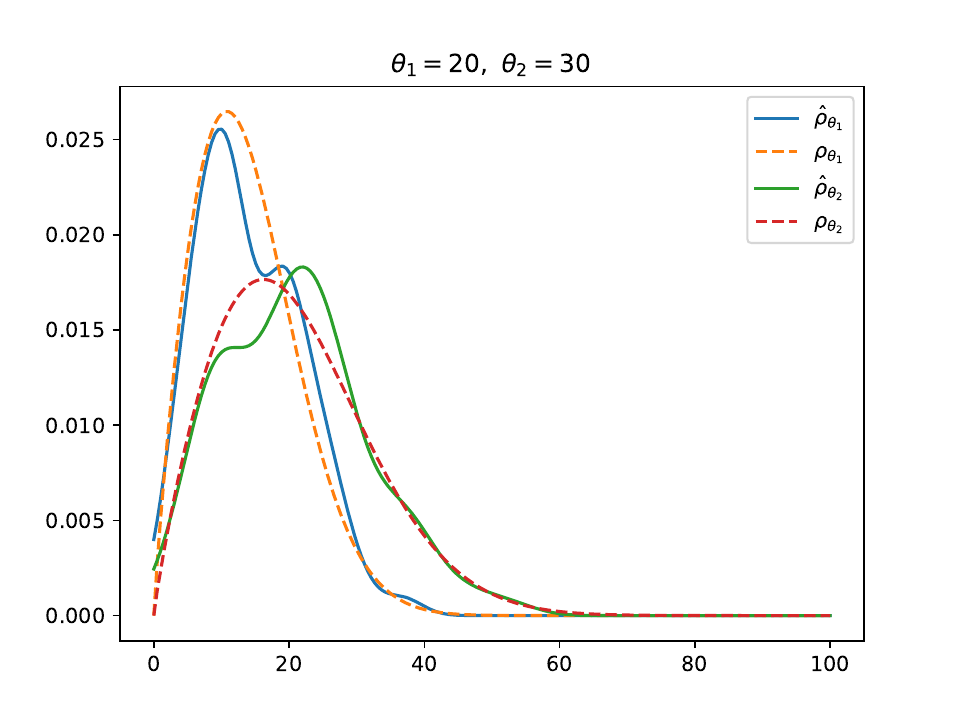}
		\caption{Spectral density estimation with $\theta_1=20$ (blue and orange)  and $\theta_2=30$ (green and red). Solid lines are the estimates.}
	\end{subfigure}
	\hfil
	\begin{subfigure}{0.325\textwidth}
		\includegraphics[width=\textwidth]{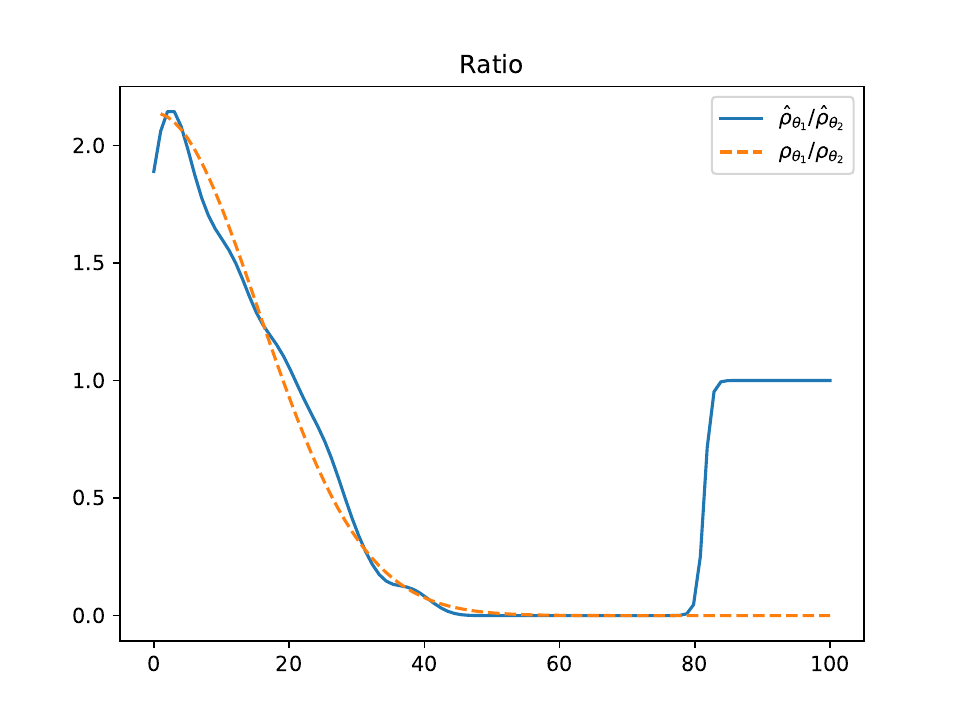}
		\caption{True ratio $\rho_{\theta_1}/\rho_{\theta_2}$ (orange dashed line) and estimated ratio $\hat\rho_{\theta_1}/\hat\rho_{\theta_2}$ (blue line) over the interval $[0,100]$.}
	\end{subfigure}
	\hfil
	\begin{subfigure}{0.325\textwidth}
		\includegraphics[width=\textwidth]{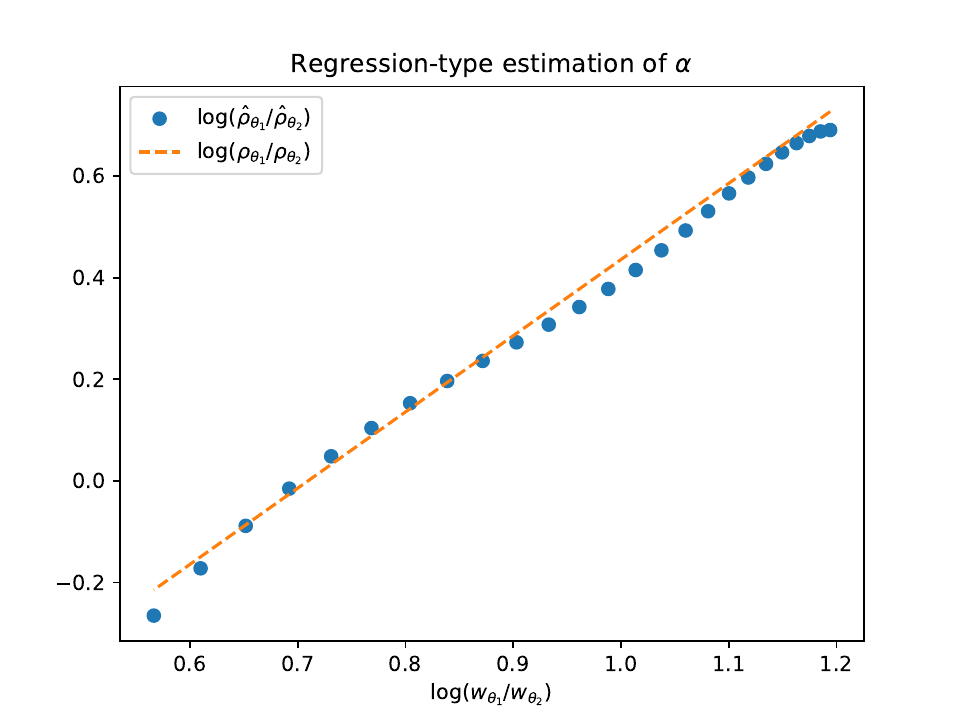}
		\caption{True log-ratio $\log(\rho_{\theta_1}/\rho{\theta_2})$ (orange dashed line) and estimated log-ratio $\log(\hat{\rho}_{\theta_1}/\hat{\rho}_{\theta_2})$ (blue dots).}
	\end{subfigure}
	\caption{Estimation of the index of stability $\alpha$ based on a single path of $X^H$ with $\alpha =1.5$ sampled equidistantly with $\delta = 0.01$, $n=10000$. The dashed lines (orange and red)  are the true normalized spectral densities $\rho_{\theta_i}$, $i=1,2$, whereas the solid lines (blue and green) are the estimates $\hat\rho_{\theta_i}$, $i=1,2$.}
	\label{fig: regression type alpha estimation}
\end{figure}

In Table \ref{table2} we give the mean bias of the regression estimator  computed from 1000 estimations of $\alpha$ based on 1000 single paths.
The accuracy of the estimator increases with increasing sample size and consequently increasing number of frequencies $N$ used for kernel density estimation. 
Naturally, the estimator performs better the better the kernel density estimation of $\rho_\theta$, which is reflected by the performance in the case $\alpha=1.5$ compared to $\alpha=0.75$. 
Note that for the computation of the mean bias, estimates $\hat\alpha\not\in(0,2)$ were excluded, see \ref{appendix tabfig} Table \ref{table2_add}. 
For instance in the case $\alpha=1.5$ with smoothing by $v^{(1)}$, sample size $n=10^4$ and number of estimated frequencies $N=200$ the number of outliers was 38 out of 1000. 
Upon closer inspection, in each of those cases high fluctuations at the tail of the estimated ratio $\hat\rho_{\theta_1}(z)/\hat\rho_{\theta_2}(z)$ were the main problem. 
A more conservative choice of the lower bound $b_l$ and the threshold $\varepsilon$ or manually adjusting the bounds of the interval on which the log ratio is computed will lead to a significant improvement of the estimate $\hat\alpha$.

\begin{table}[h]
	\centering
	\renewcommand{\arraystretch}{1.25}
	\setlength{\tabcolsep}{3pt}
	\begin{tabular}{|c|c|c||ccc|ccc|}
		\hline
		\multirow{4}{*}{\shortstack{Mean bias \\$\hat{\alpha}-\alpha$}}&\multirow{4}{*}{$\alpha$}&$\delta$&\multicolumn{6}{c|}{$0.01$}\\
		\cline{3-9}
		&&$n$&\multicolumn{3}{c|}{$10^3$}&\multicolumn{3}{c|}{$10^4$}\\
		\cline{3-9}
		&&\multirow{2}{*}{\diagbox{$H$}{$N$}}&\multirow{2}{*}{10}&\multirow{2}{*}{25}&\multirow{2}{*}{40}&\multirow{2}{*}{100}&\multirow{2}{*}{150}&\multirow{2}{*}{200}\\
		&&&&&&&&\\
		\hline\hline
		\multirow{4}{*}{$w^{(1)}_{20}$}&\multirow{2}{*}{$0.75$}&$0.25$&-0.056&0.136&-0.078&0.084&0.145&0.158\\
		&&$0.75$&-0.142&0.100&-0.101&0.081&0.109&0.143\\
		\cline{2-9}
		&\multirow{2}{*}{$1.5$}&$0.25$&-0.338&-0.221&-0.577&-0.009&-0.001&-0.032\\
		&&$0.75$&-0.467&-0.208&-0.615&-0.042&-0.003&-0.032\\
		\hline
		\multirow{4}{*}{$w^{(2)}_{20}$}&\multirow{2}{*}{$0.75$}&$0.25$&0.021&0.163&0.028&0.115&0.145&0.180\\
		&&$0.75$&0.041&0.158&0.003&0.113&0.161&0.173\\
		\cline{2-9}
		&\multirow{2}{*}{$1.5$}&$0.25$&-0.537&-0.259&-0.490&-0.161&-0.104&-0.065\\
		&&$0.75$&-0.483&-0.238&-0.529&-0.095&-0.045&-0.053\\
		\hline
	\end{tabular}
	\caption{Mean bias of the regression estimator $\hat\alpha$ to $\alpha$ based on on the log-transformed ratio $\hat\rho_{\theta_1}/\hat\rho_{\theta_2}$ with $\theta_1=20$, $\theta_2 = 30$. Computations are based on $1000$ single path simulations and spectral density estimation for each path. Each path is sampled equidistantly with mesh size $\delta = 0.01$ and sample size $n=1000$ and $n=10000$, respectively. }
	\label{table2}
\end{table}

\subsection{Analysis of the kernel density estimator $\hat\rho_\theta$}

The results in Figure \ref{fig: spectral_est_1} are based on a single path of a real HFSM, with $(\alpha,H)=(0.75,0.25)$ in Subfigures (a), (b) and once with  $(\alpha,H)=(1.5,0.75)$ in Subfigures (c), (d). 
Each path was sampled equidistantly with mesh size $\delta = 0.01$ and sample size $n=1000$ in (a), (c) and $n=10000$ in (b), (d). 
Convolution was then performed with $v^{(1)}$ of Example \ref{numex 1}, where $L=100$ and $\theta=20$ resulting in the sample $\tilde{x}=(\tilde{x}_{\theta,0},\dots,\tilde{x}_{\theta, n-L})$.
Choose the number of frequencies $N$ estimated from the periodogram of $\tilde{x}$ such that $N^{2/5}/(n-L)\leq\varepsilon$ for some small error $\varepsilon>0$. Here, for $n-L=900$, we set $N=25$ and $\varepsilon = 0.005$ for (a), (c). See \cite{viet2} for a more detailed account on the choice of $N$. 

When increasing the sample size to $n=10000$, the relation $N^{2/5}/(n-L)\leq\varepsilon$ allows us to increase the number of frequencies to be estimated to $N=150$, while keeping everything else unchanged. 
This leads to an improvement of the estimation results, see Subfigures (b) and (d).
For the kernel density estimator $\hat\rho_\theta$ we use the Gaussian kernel function with Silverman's rule for bandwidth selection.
The blue line shows the estimator $\hat{\rho}_\theta$ and the orange dashed line shows the true normalized spectral density $\rho_\theta$ of the mollified process $\tilde{X}_\theta$. 
\begin{figure}[h]
	\centering
	\begin{subfigure}{0.49\textwidth}
		\includegraphics[width=\textwidth]{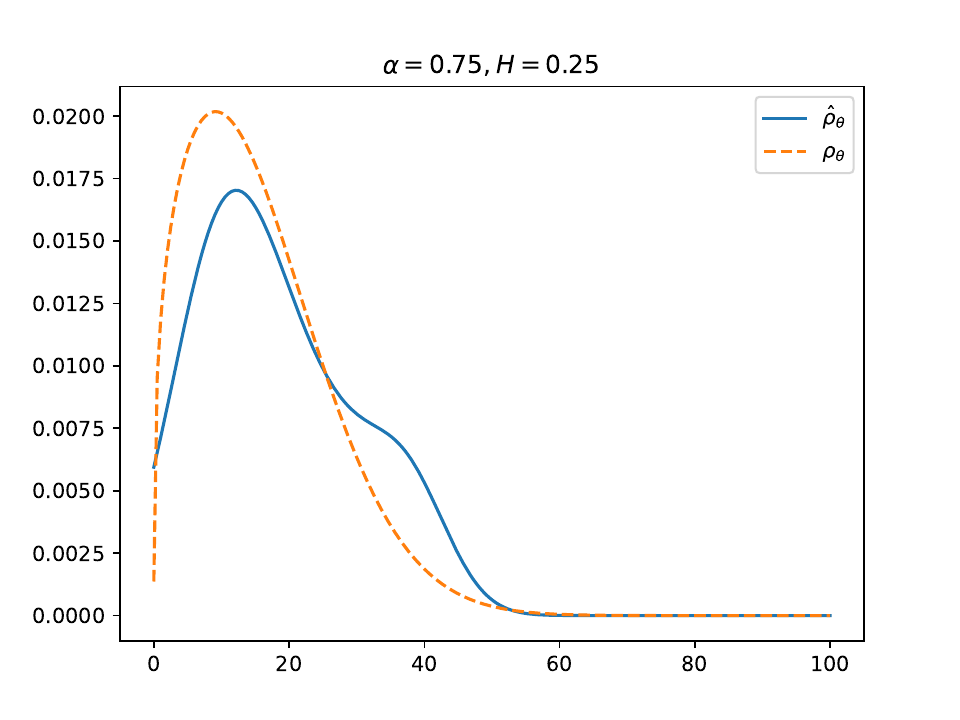}
		\caption{$(\alpha,H) = (0.75,0.25), n=10^3$}
	\end{subfigure}
	\hfill
	\begin{subfigure}{0.49\textwidth}
		\includegraphics[width=\textwidth]{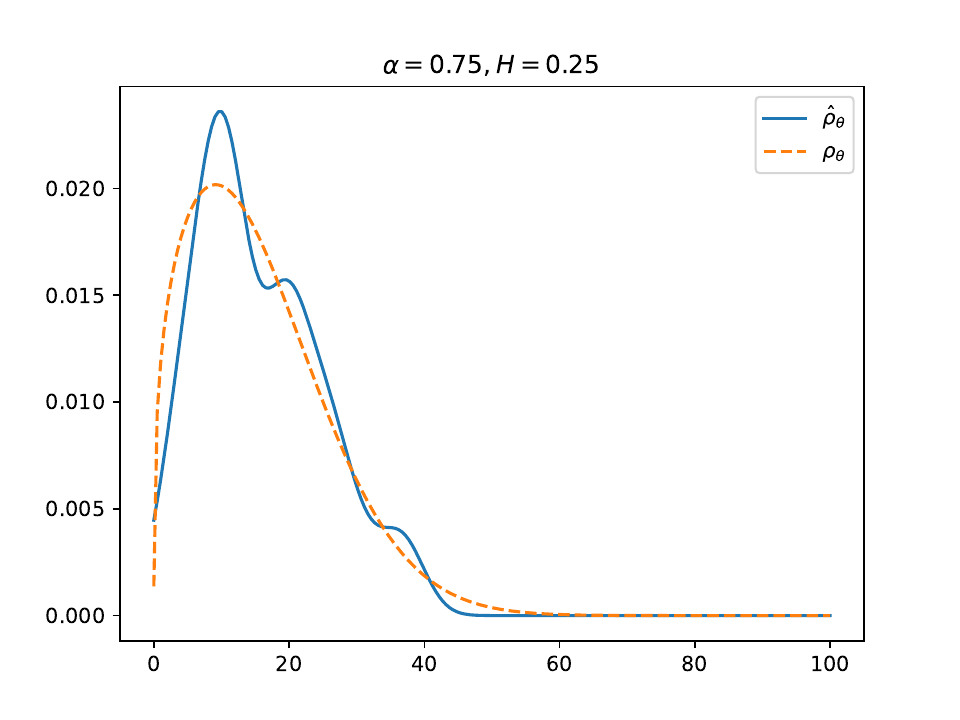}
		\caption{$(\alpha,H) = (0.75,0.25), n=10^4$}
	\end{subfigure}
	\hfill
	\begin{subfigure}{0.49\textwidth}
		\includegraphics[width=\textwidth]{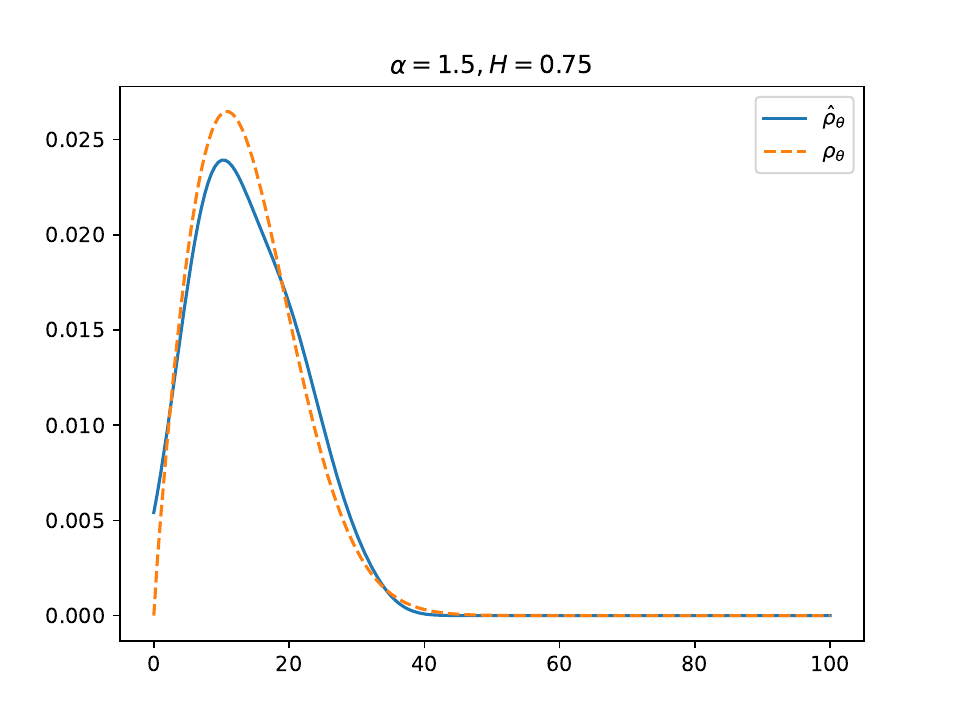}
		\caption{$(\alpha,H) = (1.5,0.75), n=10^3$}
	\end{subfigure}
	\hfil
	\begin{subfigure}{0.49\textwidth}
		\includegraphics[width=\textwidth]{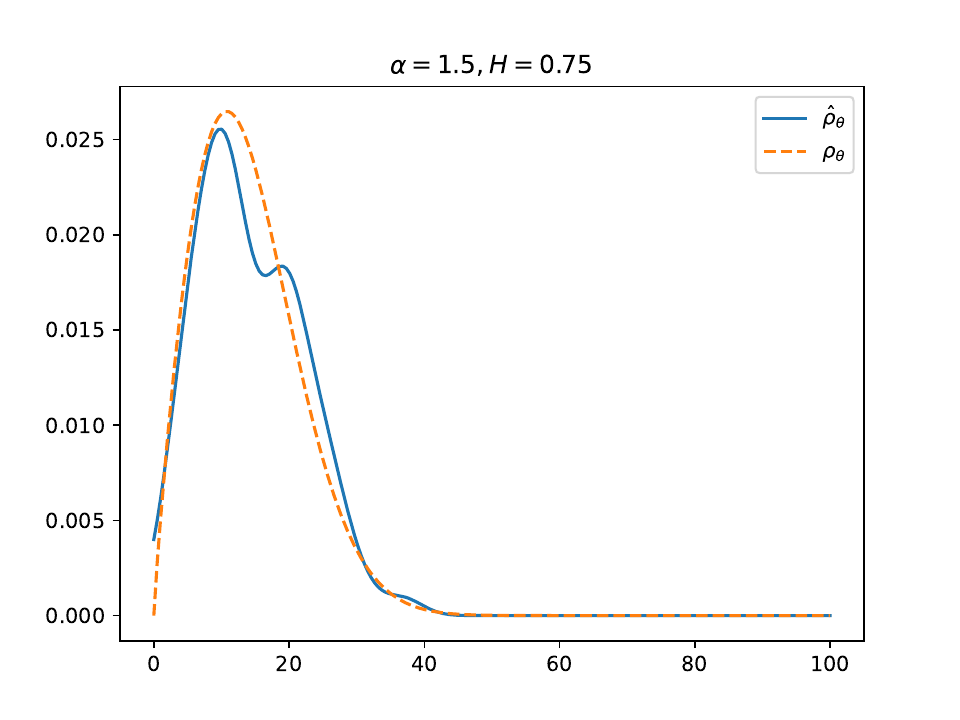}
		\caption{$(\alpha,H) = (1.5,0.75), n=10^4$}
	\end{subfigure}
	\caption{Spectral density estimation for the mollified process for the cases $(\alpha,H)=(0.75,0.25)$ in Subfigures (a) and (b) as well as $(\alpha=1.5,0.75)$ in (c) and (d), each based on a single path of a real HFSM $X^H$, respectively.}
	\label{fig: spectral_est_1}
\end{figure}

\begin{table}[h]
	\centering
	\renewcommand{\arraystretch}{1.25}
	\setlength{\tabcolsep}{3pt}
	\begin{tabular}{|c|c|c||ccc|ccc|}
		\hline
		\multirow{4}{*}{\shortstack{Mean squared $L^2$-dist.\\$\Vert\hat{\rho}-\rho\Vert_2^2$ ($\times 10^{-3}$)}}&\multirow{4}{*}{$\alpha$}&$\delta$&\multicolumn{6}{c|}{$0.01$}\\
		\cline{3-9}
		&&$n$&\multicolumn{3}{c|}{$10^3$}&\multicolumn{3}{c|}{$10^4$}\\
		\cline{3-9}
		&&\multirow{2}{*}{\diagbox{$H$}{$N$}}&\multirow{2}{*}{10}&\multirow{2}{*}{25}&\multirow{2}{*}{40}&\multirow{2}{*}{100}&\multirow{2}{*}{150}&\multirow{2}{*}{200}\\
		&&&&&&&&\\
		\hline\hline
		\multirow{4}{*}{$w^{(1)}_{20}$}&\multirow{2}{*}{$0.75$}&$0.25$&7.434&1.853&1.146&7.324&4.629&3.170\\
		&&$0.75$&7.687&2.656&2.690&7.662&4.866&3.493\\
		\cline{2-9}
		&\multirow{2}{*}{$1.5$}&$0.25$&2.499&0.690&1.213&0.922&0.515&0.357\\
		&&$0.75$&2.449&0.701&1.383&0.908&0.516&0.374\\
		\hline
		\multirow{4}{*}{$w^{(2)}_{20}$}&\multirow{2}{*}{$0.75$}&$0.25$&7.670&2.276&1.185&7.546&4.943&3.564\\
		&&$0.75$&7.834&2.184&1.097&7.637&4.965&3.569\\
		\cline{2-9}
		&\multirow{2}{*}{$1.5$}&$0.25$&2.384&0.673&1.136&0.931&0.514&0.358\\
		&&$0.75$&2.355&0.667&1.133&0.921&0.509&0.350\\
		\hline
	\end{tabular}
	\caption{Mean squared $L^2$-distances between kernel density estimators $\hat\rho_\theta$ and normalized spectral density $\rho_\theta$ over the interval $[0,100]$. Computations are based on $1000$ single path simulations and spectral density estimation for each path.. Each path is sampled equidistantly with mesh size $\delta = 0.01$ and sample sizes $n=1000$ and $n=10000$, respectively. }
	\label{table1}
\end{table}

In Table \ref{table1} we compute the mean squared $L^2$-distance $\Vert \hat\rho_{\theta}-\rho_\theta\Vert_2^2$ on the interval $[0,100]$. For this a single path of $X^H$ with $\alpha=0.75,1.5$ and $H=0.25,0.75$ was simulated 1000 times and sampled equidistantly with mesh size $\delta=0.01$.
For the sample size we consider the cases $n=1000$ and $n=10000$. 
Each sample path is smoothed with $v^{(i)}$, $i=1,2$,  from Example \ref{numex 1} and $\theta=20$, $L=100$. 
Then, $N$ frequencies are estimated from the periodogram of the sample path of the mollified process. 

We see an improvement in accuracy, the higher the sample size $n$ and the number of frequencies $N$ used for kernel density estimation are.
The smaller the index of stability and the Hurst parameter, the more difficult the estimation of $\rho$.  
This is to be expected as the path behavior becomes more erratic with decreasing $\alpha$ and $H$.
Upon further inspection we discover that the high values for the $L^2$-distance appearing in particular for the cases $n=1000$ and $\alpha=0.75$ arise from outliers in the estimation of $\rho$. 
Excluding the effect of these outliers we give the median $L^2$-distance under the same setting as Table \ref{table1} in  \ref{appendix tabfig}, Table \ref{table1_median}.

Lower accuracy of the spectral density estimation for the case $\alpha=0.75$ compared to $\alpha=1.5$ can be explained by the following.
The deviations between $\rho_\theta$ and $\hat\rho_\theta$ are due to the fact that the smaller $\alpha$ the faster the decay of the sequence $\{\Gamma_k^{-1/\alpha}\}_{k=1}^\infty$ in the LePage series representation of $\tilde{X}_\theta$, see Corollary \ref{corollary hfsm mollified}. 
As a consequence, only a few amplitudes $R_k$, see Equation \eqref{amplitudes R}, dominate all the remaining amplitudes in size. 
Due to the nature of the discrete Fourier transform, which the periodogram is based on, the periodogram will oscillate around the frequencies corresponding to the dominating amplitudes, which leads to difficulties in the estimation of the absolute frequencies $\vert Z_k\vert$. These oscillations are also called sidelobes and are addressed in \cite{hannan2,viet2,hannan}. 
 
In particular, the case $(\alpha,H)=(0.75,0.75)$ has the problem that the true density $\rho_\theta$ does not vanish at the origin.
In general, for a real HFSM this is the case when $\alpha(p-H)-1<0$, where $p\geq 1$ is the exponent in $w(x)=\vert x\vert^pe^{-\vert x\vert^q}$, which might happen if $\alpha$ is relatively small. 
Therefore, it is more likely that the random frequencies $Z_k$ take values in a neighborhood of $0$ compared to the case of $\rho_\theta$ vanishing at the origin. 
This can cause problems for the frequency estimation as it is hard to detect frequencies close to $0$ via the periodogram \cite[Chapter 3.5]{hannan}.
To counter this, consider increasing the rate of decay of $w$ at the origin by increasing $p$ as for example $p=4$ in the function $w^{(2)}$, see Example \ref{numex 1} and the corresponding smoothing function $v^{(2)}$. 

Figure \ref{fig: small alpha} compares estimation results for $\rho_\theta$ based on the same path realization of $X^H$ with $(\alpha,H)=(0.75,0.75)$ but once mollified with $v^{(1)}$ in Subfigure (a) and once with $v^{(2)}$ in Subfigure (b).
In Subfigure (a) with $p=2$ it holds that $\rho_\theta(z)\sim \vert z\vert^{-1/4} $ as $z\rightarrow 0$, where as in the case of Subfigure (b) we have $\rho_\theta(z)\sim \vert z\vert^{3/4}$ as $z\rightarrow 0$. 
In Subfigure (a) we see that the estimator $\hat{\rho}_\theta$ strongly deviates from $\rho_\theta$ close to the origin due to the aforementioned issue of small frequency estimation. 
This is not apparent in Subfigure (b) anymore. 
\begin{figure}[h]
	\centering
	\begin{subfigure}{0.49\textwidth}
		\includegraphics[width=\textwidth]{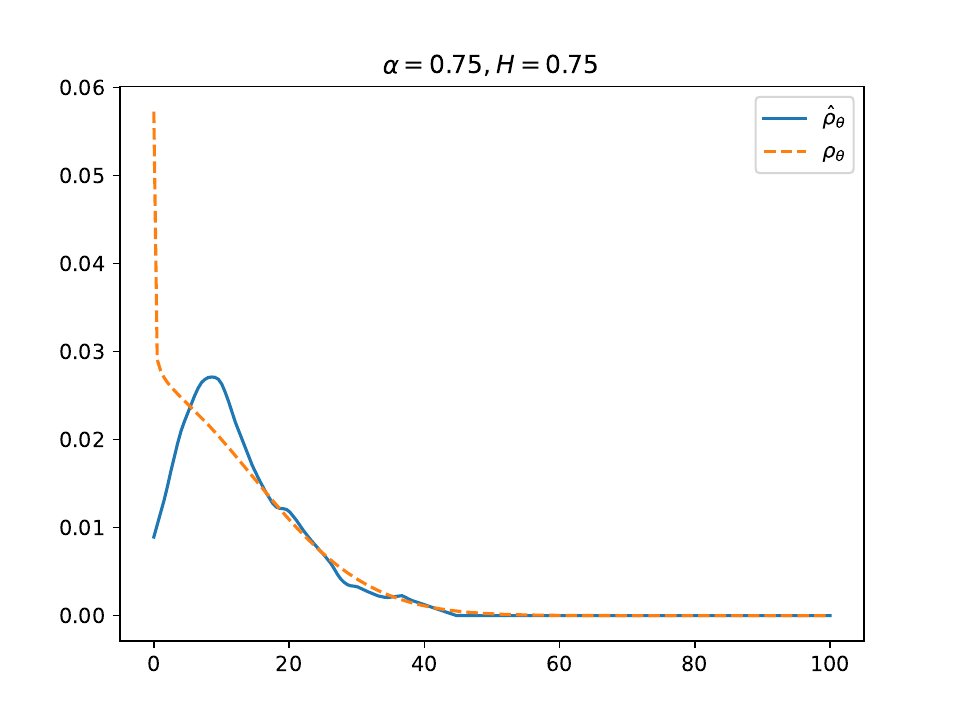}
		\caption{Smoothed by $w^{(1)}, v^{(1)}$}
	\end{subfigure}
	\hfill
	\begin{subfigure}{0.49\textwidth}
		\includegraphics[width=\textwidth]{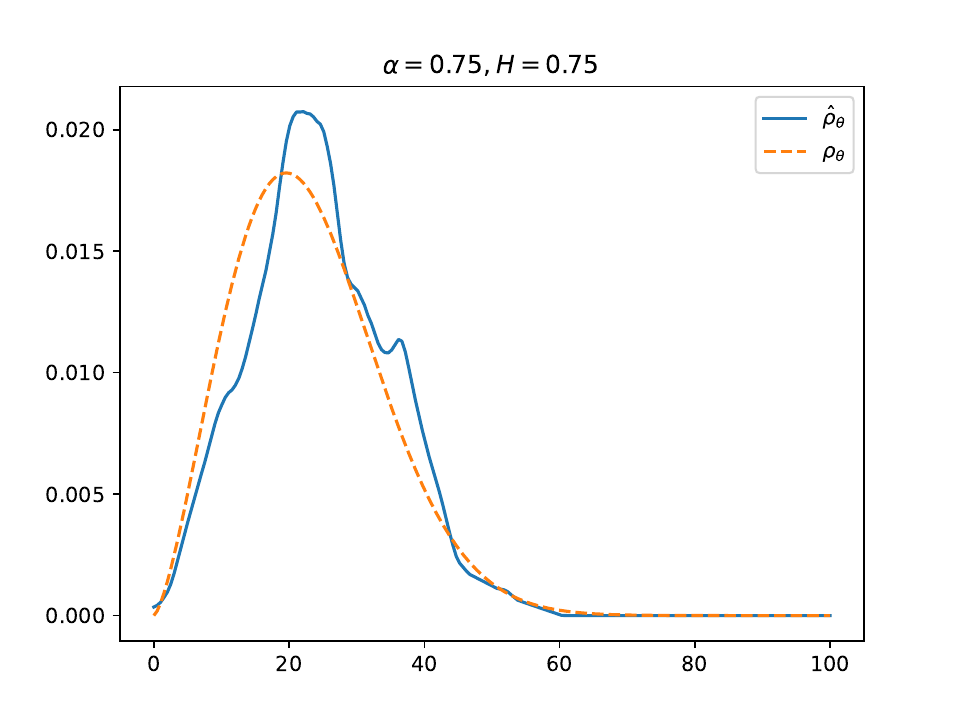}
		\caption{Smoothed by $w^{(2)}, v^{(2)}$}
	\end{subfigure}
	\caption{Spectral density estimation for the mollified process for the cases $(\alpha,H)=(0.75,0.75)$. Estimation is based on a single path realization of $X^H$ equidistantly sampled with $\delta = 0.01$ and $n=10000$. }
	\label{fig: small alpha}
\end{figure}

\subsection{Gamma distribution based estimation of $\alpha$ and $H$}
We now consider the case when it is explicitly known that the sample path belongs to a real HFSM. 
As mentioned before, errors in the estimates $\hat\rho_{\theta_i}$, $i=1,2$, directly translate to the regression estimator for $\alpha$. The advantage of estimating the parameters of the Gamma distribution directly, from which we derived estimators for the stability index $\alpha$ and Hurst parameter $H$ in Corollary \ref{corollary hfsm mollified} and \ref{cor: hfsm consistent alpha and H}, lies in avoiding the need of  kernel density estimation and the errors that come along with it as mentioned above. 
On the other hand, for the parameter estimates of the Gamma distribution to work well a large sample size $N$ of estimated frequencies $\{\hat{Z}_{k,n}\}_{k=1}^N$ is needed, which in turn requires the sample size $n$ of path sample points to be large. 
Therefore, for the next computations we only consider the case $n=10000$ with $N=100,150,200$. 

Table \ref{table3} shows the mean bias of the estimators $\tilde{\alpha}$ and $\tilde{H}$ based on Gamma distribution parameter estimation.
Again, we consider 1000 simulations of a single path of a real HFSM $X^H$ with $\alpha \in \{0.75,1.5\}$ and $H\in\{0.25,0.75\}$, sampled equidistantly with mesh size $\delta=0.01$ and sample size $n=10000$. 
The paths are smoothed with $v^{(i)}$, $i=1,2$, of Example \ref{numex 1} and $N = 100, 150, 200$ frequencies of the underlying LePage series are estimated using the periodogram of the mollified sample.  
\begin{table}[h]
	\centering
	\renewcommand{\arraystretch}{1.25}
	\setlength{\tabcolsep}{3pt}
	\begin{tabular}{|c|c|c||ccc|ccc|}
		\hline
		\multirow{4}{*}{\shortstack{Mean bias \\$\tilde{\alpha}-\alpha$,\\$\tilde{H}-H$}}&\multirow{4}{*}{$\alpha$}&$\delta$&\multicolumn{6}{c|}{$0.01$}\\
		\cline{3-9}
		&&&\multicolumn{3}{c|}{$\tilde{\alpha}$}&\multicolumn{3}{c|}{$\tilde{H}$}\\
		\cline{3-9}
		&&\multirow{2}{*}{\diagbox{$H$}{$N$}}&\multirow{2}{*}{100}&\multirow{2}{*}{150}&\multirow{2}{*}{200}&\multirow{2}{*}{100}&\multirow{2}{*}{150}&\multirow{2}{*}{200}\\
		&&&&&&&&\\
		\hline\hline
		\multirow{4}{*}{$w^{(1)}_{20}$}&\multirow{2}{*}{$0.75$}&$0.25$&0.262&0.255&0.251&0.277&0.272&0.263\\
		&&$0.75$&0.290&0.288&0.285&-0.143&-0.137&-0.139\\
		\cline{2-9}
		&\multirow{2}{*}{$1.5$}&$0.25$&0.070&0.094&0.058&0.150&0.102&0.051\\
		&&$0.75$&0.066&0.086&0.054&0.012&0.020&-0.006\\
		\hline
		\multirow{4}{*}{$w^{(2)}_{20}$}&\multirow{2}{*}{$0.75$}&$0.25$&0.255&0.244&0.236&0.252&0.261&0.257\\
		&&$0.75$&0.281&0.283&0.259&-0.172&-0.147&-0.171\\
		\cline{2-9}
		&\multirow{2}{*}{$1.5$}&$0.25$&0.081&0.104&0.061&0.198&0.152&0.096\\
		&&$0.75$&0.093&0.099&0.067&-0.072&-0.051&-0.056\\
		\hline
	\end{tabular}
	\caption{Mean bias of the estimators $\hat\alpha$ to $\alpha$ based on on the log-transformed ratio $\hat\rho_{\theta_1}/\hat\rho_{\theta_2}$ with $\theta_1=20$, $\theta_2 = 30$. Computations are based on $1000$ single path simulations and spectral density estimation for each path, respectively. Each path is sampled equidistantly with mesh size $\delta = 0.01$ and sample size $n=1000$ and $n=10000$, respectively. }
	\label{table3}
\end{table}

The estimators perform well in the case $\alpha = 1.5$ for all $H\in\{0.25,0.75\}$. 
As expected, the bias of $\tilde{\alpha}$ and $\tilde{H}$ decreases with increasing number of frequencies $N$ fed to the estimators. 
For small $\alpha$, i.e. $\alpha=0.75$ in our case, we see that both estimators for $\alpha$ and $H$ have difficulty assessing their true counterparts. 
This is mainly due to the same issue we encounter in spectral density estimation for small $\alpha$, meaning that a handful amplitudes are dominating over the rest since the sequence $\{\Gamma_k^{-1/\alpha}\}_{k=1}^\infty$ decays faster for smaller $\alpha$. 

\begin{rem}
	\label{rem: choice p q}
	By \cite[Theorem 3.2]{gammadistr} and the delta method it holds that the estimators $\tilde{b},\tilde{r}$ are asymptotically normal with asymptotic covariance matrix 
	\begin{align*}
		\Sigma = \left(\begin{matrix}
			\sigma_b^2&\sigma_{br}\\\sigma_{br} & \sigma_r^2
		\end{matrix}\right),
	\end{align*}
	where 
	\begin{align*}
		\sigma_b^2 = b^2\left(1+b\Psi_1(1+b)\right),\quad
		\sigma_r^2 = r^2\left(1+b\Psi_1(b)\right),\quad
		\sigma_{br} =  -b r \left(1+b\Psi_1(1+b)\right),
	\end{align*}
	and $\Psi_1(x)=d^2/dx^2\log(\Gamma(x))$ is the trigamma function. 
	Aiming to minimize the Frobenius norm of $\Sigma$, given by 
	\begin{align*}
		\Vert\Sigma\Vert = \sqrt{b^4\left(1+b\Psi_1(1+b)\right)^2
			+r^4\left(1+b\Psi_1(b)\right)^2
			+2b^2 r^2 \left(1+b\Psi_1(1+b)\right)^2
		},
	\end{align*}
	we first see that it is monotonically increasing in $b,r$.
	From Corollary \ref{corollary hfsm mollified} we know that 
	\begin{align*}
		b = \frac{\alpha}{q}(p-H),\quad r  = \frac{\alpha}{\theta^q}, \qquad p,q\geq 1,
	\end{align*}
	which suggests that, in theory, $p\geq 1$ should be chosen as small as possible, i.e. $p=1$, and $q\geq1$ as large as possible, leading to no immediate practical choice of $p,q$.
	
	In practice however, due to problems in estimating frequencies close to $0$, the condition $\alpha(p-H)-1>0$ is favorable so that the spectral density $\rho_\theta$ of the mollified process $\tilde{X}_\theta$ in Corollary \ref{corollary hfsm mollified} vanishes at the origin, cf. Figure \ref{fig: small alpha}. This indicates the lower bound $p\geq \max\{1,1/\alpha + H\}$. 
	A more conservative bound is given by $p\geq 1/\alpha+1$, since $H\in(0,1)$.
	On the other hand, choosing $p$ too conservatively, i.e. too large, results in fluctuations of $v=\mathcal{F}^{-1}w$ at its tails, see Figure \ref{fig: pq} (a).
	
	Furthermore, concerning the choice of the parameter $q\geq 1$, note that $q=2$ allows for the explicit computation of the smoothing kernel $v=\mathcal{F}^{-1}w$, see Example \ref{numex 1}. 
	In general, the function $v$ cannot be given for $q>2$. 
	Exceptions are the cases with even integers $q>2$, for which $v$ can be expressed as a sum of generalized hypergeometric functions. 
	Numerical computations show that for $w(x)=x^2e^{-x^q}$, $q>2$, the function $v=\mathcal{F}^{-1}w$ exhibits stronger fluctuations at its tails and slower absolute decay compared to the case $q=2$, compare Figure \ref{fig: pq} (b).
	
	Recall that $v_\theta$ needs to be discretized at the points $\{\delta(-L/2+j)\}_{j=0}^L$ for the discrete convolution in the computation of $\tilde{X}_\theta$.
	The mesh size $\delta$ and the number of convolution point $L$ (which depends on the sample size $n$ by $n-L>n_0$ with minimal sample size $n_0$ for the mollified process) dictate the choice of the parameters $\theta$, $p$ and $q$ such that $\vert v_\theta(t)\vert<\varepsilon$ for $t\not\in [-\delta L/2,\delta L/2]=I$ and some small error $\varepsilon>0$. Additionally,
	the quadrature error of the discrete convolution \eqref{eq: discrete convolution} should be minimal. 
	
	Let $v\in C^{k+2}([a,b])$ be an arbitrary function, and denote by $I_k(v)$, $k\in\mathbb{N}_0$, the \emph{Newton-Cotes quadrature} for equidistant discretization points, which generalizes common quadrature formulae such as the rectangle ($k=0$), trapezoid($k=1$) or Simpson formula ($k=2$). The exact integral of $v$ over $[a,b]$ is denoted by $I(v)$. Then, the quadrature error for the discretization of $v_\theta(x)=v(\theta x)$ is given by
	\begin{align*}
		E_k(v) = \left\vert I(v)-I_k(v)\right\vert = \begin{cases}
			C_k v_\theta^{(k+2)}(\xi)=\theta^{k+2}C_k v^{(k+2)}(\xi),& k\text{ even},\\
			C'_k v_\theta^{(k+1)}(\xi)=\theta^{k+1}C'_k v^{(k+1)}(\xi),&k\text{ odd},
		\end{cases}
	\end{align*}
	where $\xi\in[a,b]$ and $C_k,C'_k$ are constants given in \cite[Theorem 9.2]{numana}.
		
	Slower decay and fluctuations at the tail of $v_\theta$ for larger $p$ and $q$ require a larger choice of $\theta$ for $v_\theta$ to be negligible outside of the interval $I$. 
	The above quadrature error indicates that larger values of $\theta$ increase the quadrature error, which directly translates to the discrete convolution. 
	This is due to the fluctuations of $v_\theta$ being compressed onto the interval $I$ by scaling with $\theta$.
	The inaccuracy of the discretization can only be mitigated by a smaller mesh size $\delta$. 
	Hence, in practice the discrete convolution for the computation of $\tilde{X}$ works best with $p=2,4$ and $q=2$ under the given values of $\alpha, \delta$ and $L$.
	\begin{figure}[h]
		\centering
		\begin{subfigure}
			{0.49\textwidth}
			\includegraphics[width=\textwidth]{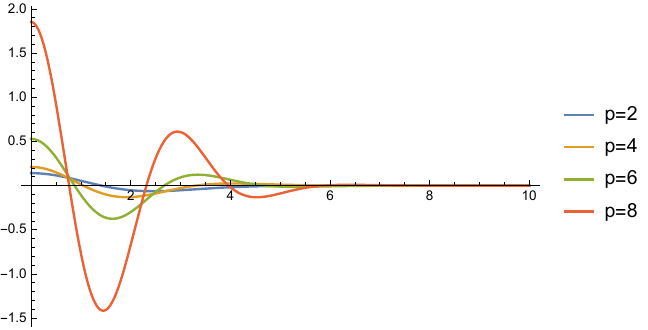}
			\caption{The smoothing function $v=\mathcal{F}^{-1}w$ with $w(x)=x^pe^{-x^2}$, $p=2,4,6,8$.}
		\end{subfigure}
		\begin{subfigure}{0.49\textwidth}
			\includegraphics[width=\textwidth]{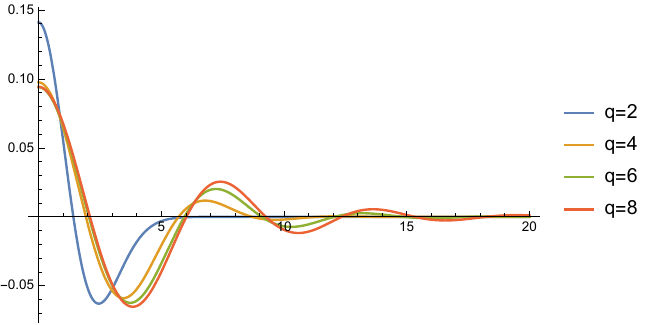}
			\caption{The smoothing function $v=\mathcal{F}^{-1}w$ with $w(x)=x^2e^{-x^q}$, $q=2,4,6,8$.}
		\end{subfigure}
		\caption{Analysis of $v=\mathcal{F}^{-1}w$ with $w(x)=x^pe^{-x^q}$, $p,q\geq 1$ even. }
		\label{fig: pq}
	\end{figure}
\end{rem}

\section{Conclusion}
\label{section: conclusion}

Real harmonizable fractional stable motions are a natural extension of the harmonizable representation of fractional Brownian motions to the stable regime. 
Opposed to linear fractional stable motions real harmonizable fractional stable motions are non-ergodic deeming standard statistical methods that rely on the law of large numbers for consistency unfeasible. 

We consider the more general class of stationary-increment harmonizable stable processes introduced in Equation \eqref{statincrementSAS} and give conditions on the integrability of the path of such a process with respect to the measure $v(t)dt$ with density $v\in L^1(\mathbb{R})\cap L^2(\mathbb{R})$. 
Exploiting the integral representation of a stationary-increment harmonizable stable process $X$ and the Fourier transform on the real line, interchanging the order of integration of the convolution $\tilde{X}_\nu(s)=\int_\mathbb{R}X(t)v(t-s)dt$ allows us to give conditions on $w = \mathcal{F}v$ in Theorem \ref{thm: mollify} such that the mollified process $\tilde{X}_\nu$ is a stationary real harmonizable $S\alpha S$ process with integrable spectral density. 

Applying periodogram frequency estimation as in \cite{viet2} yields estimates for the i.i.d. random frequencies of the LePage series representation of the mollified process $\tilde{X}$. 
Kernel density estimation allows to assess for the normalized spectral density of $\tilde{X}_\nu$ from which we derive methods to estimate the index of stability as well as the integral kernel $\Psi$ of the original stationary-increment harmonizable process $X$ up to a constant factor, see Corollary \ref{cor: consistent alpha} and Remark \ref{rem: psi}.

When considering real harmonizable fractional stable motions, the explicit form of $\Psi$ can be taken advantage of. For a specific choice of $w$ and $v$ given in Corollary \ref{corollary hfsm mollified} the mollified process has underlying random frequencies whose $p$-th absolute powers ($p$ depending on the choice of the weight function $w$) are i.i.d. Gamma distributed. The shape and scale parameters of this Gamma distribution can be estimated consistently enabling us to derive consistent estimators for the index of stability $\alpha$ as well as the Hurst parameter $H$. 

In an extensive simulation study we considered 1000 single path realizations of a real harmonizable fractional stable motion. Each path is then smoothed by example functions $v^{(1)}$ and $v^{(2)}$ with different tuning parameteres $\theta_1,\theta_2$, yielding sample paths of a mollified process, which we showed to be a stationary real harmonizable $S\alpha S$ process with finite control measure. 
We first assume that only the information that all paths are sampled from a stationary-increment harmonizable stable process is given. We estimate the stability index $\alpha$ based on a regression of the log-transformed ratio of two spectral densities from the same path and the same smoothing function $v^{(i)}$ but different tuning parameters $\theta_1,\theta_2$ as described in Section \ref{section: stat incr proc}.
From the estimates of the spectral density and index of stability it is then possible to reconstruct the kernel in the integral representation of a stationary increment harmonizable process up to a constant factor. 

Lastly, provided that the process is a real harmonizable fractional stable motion, we can estimate $\alpha$ and the Hurst parameter $H$ as described in Section \ref{section: fractionalmotion}. 
This approach eliminates problems that might arise from inaccurate kernel density estimation, but requires a large number of estimated frequencies $N$ from the periodogram, which in turn dictates larger sample sizes $n$. 

In our experiments, all estimation procedures work well for larger $\alpha$, i.e. for $\alpha = 1.5$. 
When the index of stability $\alpha$ is small, estimating the underlying random frequencies via the locations of the peaks in the periodogram is difficult. In this case a small number of amplitudes $R_k$, recall Equation \eqref{amplitudes R}, dominate the rest since the decay of $\Gamma_k^{-1/\alpha}$ as $k\rightarrow\infty$ is faster for small values of $\alpha$. 
The periodogram will exhibit strong sidelobe oscillations at the frequencies corresponding to these few dominating amplitudes. The sidelobes might not be completely filtered out during the iterative computation of the frequencies which can hinder the precise estimation of the true underlying frequencies $\vert Z_{[k]}\vert$. 
Smoothing the periodogram or considering different window functions can reduce sidelobe effects but at the cost of precision for the frequency estimation. 
Other sources of error lie in the discretization in the simulation of the samples as well as the discrete convolution. 

Summarizing, our proposed methods achieve accurate results as long as the underlying frequency estimation is accurate. 
All methods are computationally efficient since discrete convolution and periodogram computation with the fast Fourier transform are computationally inexpensive. 
Most notably, frequency estimation via the periodogram circumvents the problem of non-ergodicity. 
Results can be improved if the frequency estimation becomes more accurate. 
%

\appendix




\bibliographystyle{abbrvnat}

\bibliography{bibliography}

\begin{thebibliography}{43}
\providecommand{\natexlab}[1]{#1}
\providecommand{\url}[1]{\texttt{#1}}
\expandafter\ifx\csname urlstyle\endcsname\relax
  \providecommand{\doi}[1]{doi: #1}\else
  \providecommand{\doi}{doi: \begingroup \urlstyle{rm}\Url}\fi

\bibitem[Ayache(2023)]{ayache}
A.~Ayache.
\newblock Harmonizable fractional stable motion: simultaneous estimators for
  the both parameters.
\newblock \emph{Preprint, HAL Id: hal-04091407}, 2023.

\bibitem[Ayache and Boutard(2017)]{ayache2}
A.~Ayache and G.~Boutard.
\newblock Stationary increments harmonizable stable fields: upper estimates on
  path behaviour.
\newblock \emph{J. Theoret. Probab.}, 30\penalty0 (4):\penalty0 1369--1423,
  2017.

\bibitem[Ayache and Hamonier(2012)]{ayache3}
A.~Ayache and J.~Hamonier.
\newblock Linear fractional stable motion: a wavelet estimator of the
  {$\alpha$} parameter.
\newblock \emph{Statist. Probab. Lett.}, 82\penalty0 (8):\penalty0 1569--1575,
  2012.

\bibitem[Ayache and Hamonier(2014)]{ayache1}
A.~Ayache and J.~Hamonier.
\newblock Linear multifractional stable motion: fine path properties.
\newblock \emph{Rev. Mat. Iberoam.}, 30\penalty0 (4):\penalty0 1301--1354,
  2014.

\bibitem[Ayache and Xiao(2016)]{ayachexiao2}
A.~Ayache and Y.~Xiao.
\newblock Harmonizable fractional stable fields: local nondeterminism and joint
  continuity of the local times.
\newblock \emph{Stochastic Process. Appl.}, 126\penalty0 (1):\penalty0
  171--185, 2016.

\bibitem[Ayache et~al.(2020)Ayache, Shieh, and Xiao]{ayachexiao}
A.~Ayache, N.-R. Shieh, and Y.~Xiao.
\newblock Wavelet series representation and geometric properties of
  harmonizable fractional stable sheets.
\newblock \emph{Stochastics}, 92\penalty0 (1):\penalty0 1--23, 2020.

\bibitem[Basse-O'Connor and Podolskij(2024)]{oconnorpodolskij}
A.~Basse-O'Connor and M.~Podolskij.
\newblock Asymptotic theory for quadratic variation of harmonizable fractional
  stable processes.
\newblock \emph{Theory Probab. Math. Statist.}, \penalty0 (110):\penalty0
  3--12, 2024.

\bibitem[Basse-O'Connor et~al.(2017)Basse-O'Connor, Lachi\`eze-Rey, and
  Podolskij]{podolskij3}
A.~Basse-O'Connor, R.~Lachi\`eze-Rey, and M.~Podolskij.
\newblock Power variation for a class of stationary increments {L}\'evy driven
  moving averages.
\newblock \emph{Ann. Probab.}, 45\penalty0 (6B):\penalty0 4477--4528, 2017.

\bibitem[Beran(2017)]{beran2017statistics}
J.~Beran.
\newblock \emph{Statistics for long-memory processes}.
\newblock Routledge, 2017.

\bibitem[Bierm\'e et~al.(2007)Bierm\'e, Meerschaert, and Scheffler]{bierme}
H.~Bierm\'e, M.~Meerschaert, and H.-P. Scheffler.
\newblock Operator scaling stable random fields.
\newblock \emph{Stochastic Process. Appl.}, 117\penalty0 (3):\penalty0
  312--332, 2007.

\bibitem[Cambanis and Maejima(1989)]{cambanismaejima}
S.~Cambanis and M.~Maejima.
\newblock Two classes of self-similar stable processes with stationary
  increments.
\newblock \emph{Stochastic Process. Appl.}, 32\penalty0 (2):\penalty0 305--329,
  1989.

\bibitem[Cambanis and Soltani(1984)]{cambanissoltani}
S.~Cambanis and A.~R. Soltani.
\newblock Prediction of stable processes: Spectral and moving average
  representation.
\newblock \emph{Z. Wahrscheinlichkeitstheorie verw. Gebiete}, 66\penalty0
  (4):\penalty0 563--612, 1984.

\bibitem[Cambanis et~al.(1987)Cambanis, Hardin, and Weron]{cambanis1}
S.~Cambanis, C.~D. Hardin, Jr., and A.~Weron.
\newblock Ergodic properties of stationary stable processes.
\newblock \emph{Stochastic Process. Appl.}, 24\penalty0 (1):\penalty0 1--18,
  1987.

\bibitem[Dahlhaus(1989)]{dahlhaus}
R.~Dahlhaus.
\newblock Efficient parameter estimation for self-similar processes.
\newblock \emph{Ann. Statist.}, 17\penalty0 (4):\penalty0 1749--1766, 1989.

\bibitem[Drygas(1976)]{drygas}
H.~Drygas.
\newblock Weak and strong consistency of the least squares estimators in
  regression models.
\newblock \emph{Z. Wahrscheinlichkeitstheorie und Verw. Gebiete}, 34\penalty0
  (2):\penalty0 119--127, 1976.

\bibitem[Embrechts and Maejima(2002)]{selfsimilarprocesses}
P.~Embrechts and M.~Maejima.
\newblock \emph{Self-similar processes}.
\newblock Princeton Series in Applied Mathematics. Princeton University Press,
  Princeton, NJ, 2002.

\bibitem[Fox and Taqqu(1986)]{foxtaqqu}
R.~Fox and M.~S. Taqqu.
\newblock Large-sample properties of parameter estimates for strongly dependent
  stationary {G}aussian time series.
\newblock \emph{Ann. Statist.}, 14\penalty0 (2):\penalty0 517--532, 1986.

\bibitem[Grafakos(2008)]{grafakos}
L.~Grafakos.
\newblock \emph{Classical fourier analysis}, volume~2.
\newblock Springer, New York, 2008.

\bibitem[Hannan(1973)]{hannan2}
E.~J. Hannan.
\newblock The estimation of frequency.
\newblock \emph{J. Appl. Probab.}, 10\penalty0 (3):\penalty0 510--519, 1973.

\bibitem[Hoang and Spodarev(2021)]{viet}
L.~V. Hoang and E.~Spodarev.
\newblock Inversion of $\alpha$-sine and $\alpha$-cosine transforms on
  $\mathbb{R}$.
\newblock \emph{Inverse Problems}, 37\penalty0 (8), 2021.

\bibitem[Hoang and Spodarev(2024)]{viet2}
L.~V. Hoang and E.~Spodarev.
\newblock Non-ergodic statistics and spectral density estimation for stationary
  real harmonizable symmetric $\alpha$-stable processes.
\newblock \emph{Bernoulli}, To appear, 2024.

\bibitem[Kallenberg(2002)]{kallenberg}
O.~Kallenberg.
\newblock \emph{Foundations of Modern Probability}.
\newblock Springer, 2002.

\bibitem[Karatzas and Shreve(2014)]{brownian}
I.~Karatzas and S.~Shreve.
\newblock \emph{Brownian Motion and Stochastic Calculus}.
\newblock Graduate Texts in Mathematics. Springer New York, 2014.

\bibitem[Kolmogoroff(1940)]{kolmogorovfbm}
A.~N. Kolmogoroff.
\newblock Wienersche {S}piralen und einige andere interessante {K}urven im
  {H}ilbertschen {R}aum.
\newblock \emph{C. R. (Doklady) Acad. Sci. URSS (N.S.)}, 26:\penalty0 115--118,
  1940.

\bibitem[K\^{o}no and Maejima(1991)]{maejima}
N.~K\^{o}no and M.~Maejima.
\newblock Self-similar stable processes with stationary increments.
\newblock In \emph{Stable processes and related topics ({I}thaca, {NY}, 1990)},
  volume~25 of \emph{Progr. Probab.}, pages 275--295. Birkh\"{a}user Boston,
  Boston, MA, 1991.

\bibitem[Lamperti(1962)]{lamperti1962semi}
J.~Lamperti.
\newblock Semi-stable stochastic processes.
\newblock \emph{Trans. Am. Math. Soc.}, 104\penalty0 (1):\penalty0 62--78,
  1962.

\bibitem[Ljungdahl and Podolskij(2020)]{podolskij2}
M.~M. Ljungdahl and M.~Podolskij.
\newblock A minimal contrast estimator for the linear fractional stable motion.
\newblock \emph{Stat. Inference Stoch. Process.}, 23\penalty0 (2):\penalty0
  381--413, 2020.

\bibitem[Mandelbrot and Taqqu(1979)]{mandelbrottaqqu}
B.~B. Mandelbrot and M.~S. Taqqu.
\newblock Robust {$R/S$} analysis of long-run serial correlation.
\newblock \emph{Bull. Inst. Internat. Statist.}, 48\penalty0 (2):\penalty0
  69--99, 1979.

\bibitem[Mandelbrot and Van~Ness(1968)]{fbm}
B.~B. Mandelbrot and J.~W. Van~Ness.
\newblock Fractional {B}rownian motions, fractional noises and applications.
\newblock \emph{SIAM Rev.}, 10\penalty0 (4):\penalty0 422--437, 1968.

\bibitem[Mandelbrot and Wallis(1969{\natexlab{a}})]{mandelbrotwallis1}
B.~B. Mandelbrot and J.~R. Wallis.
\newblock Computer experiments with fractional {G}aussian noises: Part 1,
  averages and variances.
\newblock \emph{Water Resour. Res.}, 5\penalty0 (1):\penalty0 228--241,
  1969{\natexlab{a}}.

\bibitem[Mandelbrot and Wallis(1969{\natexlab{b}})]{mandelbrotwallis2}
B.~B. Mandelbrot and J.~R. Wallis.
\newblock Computer experiments with fractional {G}aussian noises: Part 2,
  rescaled ranges and spectra.
\newblock \emph{Water Resour. Res.}, 5\penalty0 (1):\penalty0 242--259,
  1969{\natexlab{b}}.

\bibitem[Mandelbrot and Wallis(1969{\natexlab{c}})]{mandelbrotwallis3}
B.~B. Mandelbrot and J.~R. Wallis.
\newblock Computer experiments with fractional gaussian noises: Part 3,
  mathematical appendix.
\newblock \emph{Water Resour. Res.}, 5\penalty0 (1):\penalty0 260--267,
  1969{\natexlab{c}}.

\bibitem[Mazur et~al.(2020)Mazur, Otryakhin, and Podolskij]{podolskij}
S.~Mazur, D.~Otryakhin, and M.~Podolskij.
\newblock Estimation of the linear fractional stable motion.
\newblock \emph{Bernoulli}, 26\penalty0 (1):\penalty0 226--252, 2020.

\bibitem[Mikosch et~al.(2002)Mikosch, Resnick, Rootz{\'e}n, and Stegeman]{fbm2}
T.~Mikosch, S.~Resnick, H.~Rootz{\'e}n, and A.~Stegeman.
\newblock Is network traffic appriximated by stable {L\'e}vy motion or
  fractional {B}rownian motion?
\newblock \emph{Ann. Appl. Probab.}, 12\penalty0 (1):\penalty0 23--68, 2002.

\bibitem[Quarteroni et~al.(2007)Quarteroni, Sacco, and Saleri]{numana}
A.~Quarteroni, R.~Sacco, and F.~Saleri.
\newblock \emph{Numerical mathematics}.
\newblock Springer-Verlag, Berlin, 2007.

\bibitem[Quinn and Hannan(2001)]{hannan}
B.~G. Quinn and E.~J. Hannan.
\newblock \emph{The Estimation and Tracking of Frequency}.
\newblock Cambridge University Press, 2001.

\bibitem[Rosinski(1995)]{rosinski}
J.~Rosinski.
\newblock On the structure of stationary stable processes.
\newblock \emph{Ann. Probab.}, 23\penalty0 (3):\penalty0 1163--1187, 1995.

\bibitem[Samorodnitsky(1992)]{integrability}
G.~Samorodnitsky.
\newblock Integrability of stable processes.
\newblock \emph{Probab. Math. Statist.}, 13\penalty0 (2):\penalty0 191--204,
  1992.

\bibitem[Samorodnitsky and Taqqu(1989)]{samorodnitskytaqquLFSM}
G.~Samorodnitsky and M.~S. Taqqu.
\newblock The various linear fractional {L}\'evy motions.
\newblock In \emph{Probability, statistics, and mathematics}, pages 261--270.
  Academic Press, Boston, MA, 1989.

\bibitem[Samorodnitsky and Taqqu(1994)]{gennady}
G.~Samorodnitsky and M.~S. Taqqu.
\newblock \emph{Stable Non-Gaussian Random Processes: Stochastic Models with
  Infinite Variance}.
\newblock Chapman and Hall, 1994.

\bibitem[Stoev et~al.(2002)Stoev, Pipiras, and Taqqu]{pipirasstoevtaqqu}
S.~Stoev, V.~Pipiras, and M.~S. Taqqu.
\newblock Estimation of the self-similarity parameter in linear fractional
  stable motion.
\newblock \emph{Signal Processing}, 82\penalty0 (12):\penalty0 1873--1901,
  2002.

\bibitem[Willinger et~al.(1998)Willinger, Paxson, and Taqqu]{fbm3}
W.~Willinger, V.~Paxson, and M.~S. Taqqu.
\newblock Self-similarity and heavy tails: Structural modeling of network
  traffic.
\newblock \emph{A practical guide to heavy tails: statistical techniques and
  applications}, 23\penalty0 (1):\penalty0 27--53, 1998.

\bibitem[Ye and Chen(2017)]{gammadistr}
Z.-S. Ye and N.~Chen.
\newblock Closed-form estimators for the {G}amma distribution derived from
  likelihood equations.
\newblock \emph{The American Statistician}, 71\penalty0 (2):\penalty0 177--181,
  2017.

\end{thebibliography}

\newpage
\section{On the non-ergodicity of stationary real harmonizable stable processes}
\label{appendix: ergodicity}

For the fundamentals on ergodic theory for stationary processes, in particular, basics such as measure-preserving transforms, invariant sets and Birkhoff's ergodic theorem, we refer to \cite[Chapter 10]{kallenberg}.
The $\sigma$-algebra of invariant sets of a stationary real harmonizable $S\alpha S$ process $X$ is given by $\sigma(\{\Gamma_k^{-1/\alpha}G_k^{(1)},\Gamma_k^{-1/\alpha}G_k^{(2)},Z_k\}_{k\in\mathbb{N}})$, where the sequences of random variables $\{\Gamma_k\}_{k\in\mathbb{N}}$, $\{G_k^{(j)}\}_{k\in\mathbb{N}}$, $j=1,2$, and $\{Z_k\}_{k\in\mathbb{N}}$ are the same as in the LePage series representation \eqref{SRH-LePage} \cite[Theorem 3]{viet2}.

From its series representation \eqref{SRH-LePage} it becomes immediately apparent that the process $X$ is stationary Gaussian conditional on the sequences $\{\Gamma_k\}_{k\in\mathbb{N}}$ and $\{Z_k\}_{k\in\mathbb{N}}$. 
In fact, conditional on $\{\Gamma_k\}_{k\in\mathbb{N}}$ and $\{Z_k\}_{k\in\mathbb{N}}$, it is a so-called Gaussian harmonic process, which is a series of sinusoidal waves with (conditional) Gaussian amplitudes and frequencies $\{Z_k\}_{k\in\mathbb{N}}$ with (conditional) covariance function 
\begin{align*}
	\mathbb{E}\left[X(s)X(t)\mid \{\Gamma_k\}_{k\in\mathbb{N}},\{Z_k\}_{k\in\mathbb{N}}\right]=C_\alpha^2m(\mathbb{R})^{2/\alpha}\sum_{k=1}^\infty \Gamma_k^{-2/\alpha}\cos\left((t-s)Z_k\right),
\end{align*}
compare \cite[Proposition 6.6.4]{gennady}.
The spectral measure (from the Gaussian spectral theory, not to be mistaken for the spectral measure of a stable random vector) corresponding to the above covariance function is purely atomic and concentrated at the absolute frequencies $\{\vert Z_k\vert\}_{k\in\mathbb{N}}$.

\section{Proof of Theorem \ref{thm: mollify}}
\label{appendix: proof theorem mollify}

The proof of Theorem \ref{thm: mollify} requires the introduction of the so-called \emph{$\alpha$-sine transform} given by 
\begin{align*}
	\mathcal{T}_\alpha u (t)=\int_\mathbb{R}\left\vert\sin\left(tx\right)\right\vert^\alpha u(x)dx
\end{align*}
for $\alpha\in(0,2)$, which is well-defined for $u\in L^1(\mathbb{R},\min(1,\vert x\vert^\alpha)dx)$. 
It is easy to see that if $u$ is an even function, then $\mathcal{T}_\alpha u$ is even as well and $\mathcal{T}_\alpha u(0)=0$. 

\label{appendix: lemma T_alpha}
The following lemma describes the behavior of $\mathcal{T}_\alpha u$ at the origin and its tails, which will be useful for the proof of Theorem \ref{thm: mollify}. 
We say that two functions $f(x)$ and $g(x)$ are asymptotically equivalent as $x\rightarrow x_0$ if $\lim_{x\rightarrow x_0}f(x)/g(x)=1$ and write $f(x)\sim g(x)$ as $x\rightarrow x_0$.
\begin{lem}
	\label{lem: behavior T_alpha}
	Let $\alpha\in(0,2)$ and $u\in L^1(\mathbb{R},\min(1,\vert x\vert^\alpha)dx)$ be even. 
	Then,
	\begin{enumerate}[(i)]
		\item \label{lemma 3i}$\mathcal{T}_\alpha u(t)\sim \vert t\vert^\alpha$ as $  t \rightarrow 0$. 
		\item \label{lemma 3ii}$\mathcal{T}_\alpha \vert u\vert (t)\sim  \vert t\vert^{\alpha+\beta-1}$ as $\vert t\vert\rightarrow\infty$ for some $\beta<1$. 
		\item \label{lemma 3iii}If $u\in L^1(\mathbb{R})$, then $\mathcal{T}_\alpha u$ is bounded by $\Vert u\Vert_1$. Furthermore, $\mathcal{T}_\alpha u(t)\rightarrow c_0(\alpha)\Vert u\Vert_1$ as $\vert t\vert\rightarrow\infty$, where the constant $c_0(\alpha)$ is given by $c_0(\alpha)=2^{-\alpha}\Gamma(1+\alpha)/\Gamma(1+\alpha/2)^2$.
		
	\end{enumerate}
\end{lem}

\begin{proof}
\begin{enumerate}[(i)]
	\item Follows directly from the definition of $\mathcal{T}_\alpha$ as $\sin(t)\sim \vert t\vert$ as $ t\rightarrow0$. 
	\item Since $\mathcal{T}_\alpha u$ is even, consider just $t>0$. The substitution $y=tx$ in the definition of $\mathcal{T}_\alpha \vert u\vert$ yields 
	\begin{align*}
		\mathcal{T}_\alpha\vert u\vert (t)=\int_\mathbb{R}\vert\sin(y)\vert^\alpha\left\vert u\left(\frac{y}{t}\right)\right\vert t^{-1}dy.
	\end{align*}
	The behavior of $\mathcal{T}_\alpha \vert u\vert(t)$ as $t\rightarrow\infty$ is therefore determined by the behavior of $\vert u(x)\vert$ as $x\rightarrow 0$. 
	Since $u\in L^1(\mathbb{R},\min(1,\vert x\vert^\alpha)dx)$ needs to hold for $\mathcal{T}_\alpha u$ to be well-defined, it follows that $\vert u(x) \vert x^\alpha\sim  x ^{-\beta}$ as $x\rightarrow 0$ for some $\beta<1$, thus $\vert u(x)\vert \sim x^{-\alpha-\beta}$ as $x\rightarrow 0$.
	It follows that 
	\begin{align*}
		\left\vert u\left(\frac{y}{t}\right)\right\vert t^{-1}\sim \left(\frac{y}{t}\right)^{-\alpha-\beta} t^{-1}= yt^{\alpha+\beta-1} 
	\end{align*}
	as $t\rightarrow\infty$ and consequently as $\mathcal{T}_\alpha \vert u\vert (t)\sim t^{\alpha+\beta-1}$ as $t\rightarrow\infty$ for some $\beta<1$. 
	\item By the triangle inequality we have
	\begin{align*}
		\left\vert \mathcal{T}_\alpha u(t)\right\vert\leq\int_\mathbb{R}\left\vert\sin(tx)\right\vert^\alpha \vert u (x)\vert dx\leq\int_\mathbb{R}\left\vert u(x)\right\vert dx =\Vert u\Vert_1.
	\end{align*}
	Furthermore, for $u\in L^1(\mathbb{R})$ the $\alpha$-sine transform admits the series representation
	\begin{align*}
		\mathcal{T}_\alpha u(t)=c_0\mathcal{F}u(0)+2\sum_{k=1}^\infty c_k\mathcal{F}u(2kt),
	\end{align*}
	with coefficients $$c_k = c_k(\alpha)=\frac{(-1)^k}{2^\alpha}\frac{\Gamma(1+\alpha)}{\Gamma(1+\alpha/2-k)\Gamma(1+\alpha/2+k)},$$ see \cite[Theorem 2]{viet}. 
	Additionally, the series $\sum_{k=1}^\infty c_k$ converges absolutely to $c_0/2$ \cite[Corollary 2]{viet}. The Fourier transform $\mathcal{F}u$ is bounded by $\Vert u\Vert_1$ and vanishes at its tails.
	Also, $\mathcal{F}u(0)=\Vert u\Vert_1$ holds. 
	Hence, the dominated convergence theorem yields $\mathcal{T}_\alpha u(t)\rightarrow c_0\Vert u\Vert_1$ as $t\rightarrow\infty$ and $c_0$ can be computed from the above expression for the coefficients $c_k$, i.e. 
	\begin{align*}
		c_0 = 2^{-\alpha}\frac{\Gamma\left(1+\alpha\right)}{\Gamma\left(1+\frac{\alpha}{2}\right)^2}. 
	\end{align*}
\end{enumerate}
\end{proof}

\noindent We now turn to the proof of Theorem \ref{thm: mollify}. 
\begin{proof}[Proof of Theorem \ref{thm: mollify}]
	In general, necessary and sufficient conditions for the $p$-integrability of $\alpha$-stable processes with integral representation $X(t)=\int_{E}h_t(x)M_\alpha(dx)$ are given in \cite[Theorem 3.3]{integrability}, see also \cite[Chapter 11.3]{gennady}. 
	Setting $p=1$ and $h_t(x)=(e^{itx}-1)\Psi(x)$ yields that 
	\begin{align*}
		\int_{\mathbb{R}}\vert X(t)\vert v(t)dt<\infty\quad a.s.
	\end{align*}
	if and only if one of the following conditions is met (depending on $\alpha$):
	\begin{enumerate}[(i)]
		\item Case $\alpha\in(0,1)$:
		\begin{align}
			\label{alphaless1}
			\int_\mathbb{R}\left(\int_{\mathbb{R}}\left\vert\left(e^{itx}-1\right)\Psi(x)\right\vert v(t)dt\right)^\alpha dx&<\infty.
		\end{align}
		\item Case $\alpha\in(1,2)$:
		\begin{align}
			\label{alphagreater1}
			\int_\mathbb{R}\left(\int_{\mathbb{R}}\left\vert\left(e^{itx}-1\right)\Psi(x)\right\vert^\alpha dx\right)^{1/\alpha}v(t)dt&<\infty.
		\end{align}
		\item Case $\alpha=1$:
		\begin{align}
			\label{alpha1}
			&\int_\mathbb{R}\left(\int_{\mathbb{R}}\left\vert\left(e^{itx}-1\right)\Psi(x)\right\vert\right.\\ 
			&\hspace{5ex}\times\left.\left[1+\log_+\left(\frac{\left\vert\left(e^{itx}-1\right)\Psi(x)\right\vert\int_\mathbb{R}\int_\mathbb{R}\left\vert\left(e^{isy}-1\right)\Psi(y)\right\vert v(s)dsdy}
			{\left(\int_\mathbb{R}\left\vert\left(e^{ity}-1\right)\Psi(y)\right\vert dy\right)\left(\int_\mathbb{R}\left\vert\left(e^{isx}-1\right)\Psi(x)\right\vert v(s)ds\right)}\right)
			\right]dx\right)v(t)dt
			<\infty,\nonumber
		\end{align}
		where $\log_+(x)=\max\{0,\log(x)\}$.
	\end{enumerate}
	Rewriting
	$
	\vert(e^{itx}-1)\Psi(x)\vert=2\left\vert\sin\left(tx/2\right)\right\vert\left\vert\Psi(x)\right\vert
	$
	for all $t,x\in\mathbb{R}$ we get:
	\begin{enumerate}[(i)]
		\item Case $\alpha\in(0,1)$:
		\begin{align*}
			\int_\mathbb{R}\left(\int_{\mathbb{R}}\left\vert\left(e^{itx}-1\right)\Psi(x)\right\vert v(t)dt\right)^\alpha dx
			&=2^\alpha\int_\mathbb{R}\left(\int_\mathbb{R}\left\vert\sin\left(\frac{tx}{2}\right)\right\vert v(t)dt\right)^\alpha\left\vert\Psi(x)\right\vert^\alpha dx\\
			&=2^\alpha\int_{\mathbb{R}}\left(\mathcal{T}_1v\left(\frac{x}{2}\right)\right)^\alpha\left\vert\Psi(x)\right\vert^\alpha dx<\infty,
		\end{align*}
		since $\mathcal{T}_1v(x)\sim\vert x\vert$ as $x\rightarrow 0$, which implies $(\mathcal{T}_1v(x))^\alpha\sim\vert x\vert^\alpha$ as $x\rightarrow 0$, and $\mathcal{T}_1v(x)\rightarrow c_0(1)\Vert v\Vert_1$ as $\vert x\vert\rightarrow\infty$ for $v\in L^1(\mathbb{R})$ by Lemma \ref{lem: behavior T_alpha}(\ref{lemma 3i}) and (\ref{lemma 3iii}). 
		The finiteness of the above integral then follows from the fact that $\Psi$ lies in the space $L^\alpha(\mathbb{R},\min(1,\vert x\vert^\alpha)dx)$.
		
		\item Case $\alpha\in(1,2)$:
		\begin{align*}
			\int_\mathbb{R}\left(\int_{\mathbb{R}}\left\vert\left(e^{itx}-1\right)\Psi(x)\right\vert^\alpha dx\right)^{1/\alpha}v(t)dt
			&=2\int_\mathbb{R}\left(\int_\mathbb{R}\left\vert\sin\left(\frac{tx}{2}\right)\right\vert^\alpha\left\vert\Psi(x)\right\vert^\alpha dx\right)^{1/\alpha}v(t)dt\\
			&=2\int_\mathbb{R} \left(\mathcal{T}_\alpha\vert\Psi\vert\left(\frac{t}{2}\right)\right)^{1/\alpha}v(t)dt<\infty.
		\end{align*}
		By Lemma \ref{lem: behavior T_alpha}(\ref{lemma 3ii}) it holds that $\mathcal{T}_\alpha \vert\Psi\vert(t)\sim\vert t\vert^{\alpha+\beta-1}$, $\beta<1$, as $\vert t\vert\rightarrow\infty$, hence
		\begin{align*}
			\left(\mathcal{T}_\alpha\vert\Psi\vert(t)\right)^{1/\alpha}\sim\vert t\vert^{1+(\beta-1)/\alpha}.
		\end{align*}
		Note that $1+(\beta-1)/\alpha<1$ as $\beta<1$, thus 
		\begin{align*}
			\int_\mathbb{R}\left(\mathcal{T}_\alpha\vert\Psi\vert\left(\frac{t}{2}\right)\right)^{1/\alpha}v(t)dt<\infty
		\end{align*}
		if $v\in L^1(\mathbb{R},\max(1,\vert t\vert)dt)$. 
		This is guaranteed since $v=\mathcal{F}^{-1}w$ is continuous, $v(0)=\Vert w\Vert_1/(2\pi)$ and $v(t)=\mathcal{F}^{-1}w(t)=o(\vert t\vert^{-3})$ as $\vert t\vert\rightarrow\infty$ due to $w\in W^{1,3}(\mathbb{R})$. 
		
		\item Case $\alpha=1$: Note that monotonicity of the logarithm and the triangle inequality allows us to bound the absolute value of the double integral in Equation \eqref{alpha1} from above by replacing $v$ with $\vert v\vert$. 
		Without loss of generality $v$ can therefore be assumed to be non-negative. 
		Also, note that $1+\log_+(z)\leq 1+z$ for all $z\geq 0$,
		which allows us to bound Equation \eqref{alpha1} from above by 
		\begin{align*}
			\int_\mathbb{R}\left(\int_\mathbb{R}\vert h_t(x)\vert dx\right)v(t)dt + \int_\mathbb{R}\left(\int_\mathbb{R}\vert h_t(x)\vert\frac{\vert h_t(x)\vert\int_\mathbb{R}\left(\int_\mathbb{R}\vert h_s(y)\vert v(s)ds\right)dy}{\left(\int_\mathbb{R}\vert h_t(y)\vert dy\right)\left(\int_\mathbb{R}\vert h_s(x)\vert v(s)ds\right)}dy\right)v(t)dt	
			= I_1 + I_2.
		\end{align*}
		
		The first double integral is given by 
		\begin{align*}
			I_1=2\int_\mathbb{R}\left(\int_\mathbb{R}\left\vert\sin\left(\frac{tx}{2}\right)\right\vert\left\vert\Psi(x)\right\vert dx\right) v(t)dt=2\int_\mathbb{R}\mathcal{T}_1\vert\Psi\vert\left(\frac{t}{2}\right)v(t)dt<\infty,
		\end{align*}
		by the same argument as in the case $\alpha\in(1,2)$.
		
		We can simplify the expression inside the second double integral $I_2$ by 
		\begin{align*}
			&\vert h_t(x)\vert\frac{\vert h_t(x)\vert\int_\mathbb{R}\left(\int_\mathbb{R}\vert h_s(y)\vert v(s)ds\right)dy}{\left(\int_\mathbb{R}\vert h_t(y)\vert dy \right)\left(\int_\mathbb{R}\vert h_s(x)\vert v(s)ds\right)} 	\\
			&= 4 \left\vert\sin\left(\frac{tx}{2}\right)\right\vert^2 \vert \Psi(x)\vert^2
			\frac{2\int_\mathbb{R}\left(\int_\mathbb{R}\left\vert\sin\left(\frac{sy}{2}\right)\right\vert\vert\Psi(y)\vert v(s)ds\right)dy}{2\left(\int_\mathbb{R}\left\vert\sin\left(\frac{ty}{2}\right)\right\vert\vert\Psi(y)dy\right)\left(2\vert\Psi(x)\vert\int_\mathbb{R}\left\vert\sin\left(\frac{sx}{2}\right)\right\vert v(s)ds\right)}	\\
			&=\left\vert\sin\left(\frac{tx}{2}\right)\right\vert^2 \vert\Psi(x)\vert
			\frac{2\int_\mathbb{R}\mathcal{T}_1\vert\Psi\vert\left(\frac{s}{2}\right)v(s)ds}{\mathcal{T}_1\vert\Psi\vert\left(\frac{t}{2}\right) \mathcal{T}_1v\left(\frac{x}{2}\right)}
			=\left\vert\sin\left(\frac{tx}{2}\right)\right\vert^2 \vert\Psi(x)\vert
			\frac{I_1}{\mathcal{T}_1\vert\Psi\vert\left(\frac{t}{2}\right) \mathcal{T}_1v\left(\frac{x}{2}\right)},
		\end{align*}
		which yields
		\begin{align*}
			I_2 &= I_1\int_\mathbb{R}\left(\int_\mathbb{R}\left\vert\sin\left(\frac{tx}{2}\right)\right\vert^2\left\vert\Psi(x)\right\vert
			\frac{1}{\mathcal{T}_1\vert\Psi\vert\left(\frac{t}{2}\right)\mathcal{T}_1v\left(\frac{x}{2}\right)}dx\right)v(t)dt	\\
			&=I_1\int_\mathbb{R}\underbrace{\left(\int_\mathbb{R}\left\vert\sin\left(\frac{tx}{2}\right)\right\vert^2\left\vert\Psi(x)\right\vert \left(\mathcal{T}_1v\left(\frac{x}{2}\right)\right)^{-1} dx\right)}_{=:J(t)}\left(\mathcal{T}_1\vert\Psi\vert\left(\frac{t}{2}\right)\right)^{-1}v(t)dt.
		\end{align*}
		
		Note that 
		$
		\vert\sin\left(tx/2\right)\vert^2(\mathcal{T}_1v(x/2))^{-1}
		$
		is bounded for all $t,x\in\mathbb{R}$ and behaves asymptotically like $t^2\vert x\vert/4$ as $x\rightarrow 0$ for all $t\in\mathbb{R}$. Since $\Psi\in L^1(\mathbb{R},\min(1,\vert x\vert)dx)$, the function $J$ is well defined on $\mathbb{R}$. Moreover, it holds that $J(t)=O(t^2)$ as $t\rightarrow 0$ and for $t\rightarrow \infty$ the substitution $y=tx/2$ yields
		\begin{align*}
			J(t)=2\int_\mathbb{R}\left\vert\sin\left(y\right)\right\vert^2
			\underbrace{\left(\mathcal{T}_1v\left(\frac{y}{t}\right)\right)^{-1}}_{=O(\vert 1/t\vert^{-1})=O(\vert t\vert)}
			\underbrace{\left\vert\Psi\left(\frac{2y}{t}\right)\right\vert\frac{1}{\vert t\vert}}_{=O(\vert t\vert^{\alpha+\beta-1})=O(\vert t\vert^{\beta})} 
			dy
			=O\left(\vert t\vert^{1+\beta}\right).
		\end{align*} 
		Consequently, we have
		\begin{align*}
			J(t)\left(\mathcal{T}_1\vert\Psi\vert\left(\frac{t}{2}\right)\right)^{-1}
			=
			\begin{cases}
				O\left(t^2/\vert t\vert\right)=O\left(\vert t\vert\right),& t\rightarrow0, \\
				O\left(\vert t\vert^{1+\beta}/\vert t\vert^\beta\right)=O\left(\vert t\vert\right),& \vert t\vert\rightarrow\infty
			\end{cases}
		\end{align*}
		and 
		\begin{align*}
			\int_\mathbb{R}J(t)\left(\mathcal{T}_1\vert\Psi\vert\left(\frac{t}{2}\right)\right)^{-1}v(t)dt<\infty
		\end{align*}
		for $v\in L^1(\mathbb{R},\max(1,\vert t\vert)dt)$, which is again fulfilled due to $w\in W^{1,3}(\mathbb{R})$ as in the case $\alpha\in(1,2)$.
	\end{enumerate}
	
	Ultimately, the almost sure integrability in Equation \eqref{integrabilitycondition} follows under the given assumptions and is still guaranteed under scaling and translation of the function $v$.  
	Furthermore, it is easy to see that for the mollified process $\tilde{X}_\nu(s)=\int_\mathbb{R}X(t)v(t-s)dt$ we have 
	\begin{align*}
		\tilde{X}_\nu(s)
		=\int_\mathbb{R}Re\left(\int_\mathbb{Re}h_t(x)M_\alpha(dx)\right)v(t-s)dt
		=Re\left(\int_\mathbb{R}\left(\int_\mathbb{R}h_t(x)M_\alpha (dx)\right)v(t-s)dt\right)
	\end{align*}
	since $v$ is real valued. 
	Hence, applying \cite[Theorem 4.1]{integrability} yields
	\begin{align*}
		\tilde{X}(s)
		&\stackrel{a.s.}{=}Re\left(\int_{\mathbb{R}}\left(\int_{\mathbb{R}}\left(e^{itx}-1\right)\Psi(x)v\left(t-s\right)dt\right)M_{\alpha}(dx)\right)\\
		&=Re\left(\int_{\mathbb{R}}\Bigg(
		\underbrace{\int_{\mathbb{R}}e^{itx}v\left(t-s\right)dt}_{=e^{isx}\int_\mathbb{R}e^{itx}v( t)dt}
		-\underbrace{\int_{\mathbb{R}}v\left(t-s\right)dt}_{=w(0)=0}
		\Bigg)
		\Psi(x)M_{\alpha}(dx)\right)\\
		&=Re\left(\int_{\mathbb{R}}e^{isx}w\left(x\right)\Psi(x)M_{\alpha}(dx)\right).
	\end{align*}
	The random measure $\tilde{M}_\alpha (dx) = w\left(x\right)\Psi(x)M_{\alpha}(dx)$ is complex isotropic $S\alpha S$ with control measure $\tilde{m}(dx)=\left\vert w\left(x\right)\Psi(x)\right\vert^\alpha dx$. 
	The new control measure $\tilde{m}$ is finite since the function $w\Psi$ lies in the space $L^\alpha (\mathbb{R})$, which concludes this proof.
	
\end{proof}

\section{Additional tables}
\label{appendix tabfig}

\begin{table}[h]
	\centering
	\renewcommand{\arraystretch}{1.25}
	\setlength{\tabcolsep}{3pt}
	\begin{tabular}{|c|c|c||ccc|ccc|}
		\hline
		\multirow{4}{*}{\shortstack{Median squared $L^2$-dist.\\$\Vert\hat{\rho}-\rho\Vert_2^2$ ($\times 10^{-3}$)}}&\multirow{4}{*}{$\alpha$}&$\delta$&\multicolumn{6}{c|}{$0.01$}\\
		\cline{3-9}
		&&$n$&\multicolumn{3}{c|}{$10^3$}&\multicolumn{3}{c|}{$10^4$}\\
		\cline{3-9}
		&&\multirow{2}{*}{\diagbox{$H$}{$N$}}&\multirow{2}{*}{10}&\multirow{2}{*}{25}&\multirow{2}{*}{40}&\multirow{2}{*}{100}&\multirow{2}{*}{150}&\multirow{2}{*}{200}\\
		&&&&&&&&\\
		\hline\hline
		\multirow{4}{*}{$w^{(1)}_{20}$}&\multirow{2}{*}{$0.75$}&$0.25$&4.376&1.155&0.594&3.512&2.873&2.166\\
		&&$0.75$&5.156&1.719&2.132&4.022&3.198&2.468\\
		\cline{2-9}
		&\multirow{2}{*}{$1.5$}&$0.25$&1.236&0.510&1.126&0.446&0.278&0.214\\
		&&$0.75$&1.235&0.509&1.287&0.459&0.279&0.232\\
		\hline
		\multirow{4}{*}{$w^{(2)}_{20}$}&\multirow{2}{*}{$0.75$}&$0.25$&3.863&1.443&0.730&3.284&2.916&2.349\\
		&&$0.75$&4.301&1.400&0.674&3.317&2.829&2.341\\
		\cline{2-9}
		&\multirow{2}{*}{$1.5$}&$0.25$&1.270&0.452&1.038&0.465&0.284&0.213\\
		&&$0.75$&1.239&0.444&1.037&0.453&0.275&0.206\\
		\hline
	\end{tabular}
	\caption{Median squared $L^2$-distances between kernel density estimators $\hat\rho_\theta$ and normalized spectral density $\rho_\theta$ over the interval $[0,100]$. Same setting as in Table \ref{table1}.}
	\label{table1_median}
\end{table}

\begin{table}[h]
	\centering
	\renewcommand{\arraystretch}{1.25}
	\setlength{\tabcolsep}{3pt}
	\begin{tabular}{|c|c|c||ccc|ccc|}
		\hline
		\multirow{4}{*}{\shortstack{Number of estimates \\$\hat{\alpha}\not\in(0,2)$}}&\multirow{4}{*}{$\alpha$}&$\delta$&\multicolumn{6}{c|}{$0.01$}\\
		\cline{3-9}
		&&$n$&\multicolumn{3}{c|}{$10^3$}&\multicolumn{3}{c|}{$10^4$}\\
		\cline{3-9}
		&&\multirow{2}{*}{\diagbox{$H$}{$N$}}&\multirow{2}{*}{10}&\multirow{2}{*}{25}&\multirow{2}{*}{40}&\multirow{2}{*}{100}&\multirow{2}{*}{150}&\multirow{2}{*}{200}\\
		&&&&&&&&\\
		\hline\hline
		\multirow{4}{*}{$w^{(1)}_{20}$}&\multirow{2}{*}{$0.75$}&$0.25$&320&110&31&152&128&75\\
		&&$0.75$&344&119&53&204&143&91\\
		\cline{2-9}
		&\multirow{2}{*}{$1.5$}&$0.25$&364&28&4&159&79&29\\
		&&$0.75$&442&68&4&242&132&38\\
		\hline
		\multirow{4}{*}{$w^{(2)}_{20}$}&\multirow{2}{*}{$0.75$}&$0.25$&334&114&43&202&162&112\\
		&&$0.75$&329&112&35&179&140&90\\
		\cline{2-9}
		&\multirow{2}{*}{$1.5$}&$0.25$&396&154&2&235&186&128\\
		&&$0.75$&378&82&1&202&126&75\\
		\hline
		\hline
	\end{tabular}
	\caption{Supplementary table to Table \ref{table2}. Number of estimates $\hat\alpha\not\in(0,2)$.}
	\label{table2_add}
\end{table}

\end{document}